\documentclass[onethmnum,onefignum,onetabnum]{article} 
\usepackage{pgfplots}
\usepgfplotslibrary{units}
\usepackage{subfig}
\usepackage{a4wide}
\newcommand{\pr}{\rightarrow}

\newcommand{\ba}{\begin{array}}
\newcommand{\ea}{\end{array}}

\newcommand{\vart}{\vartheta}

\newcommand{\eps}{\varepsilon}

\newcommand{\il}{\int\limits}

\newcommand{\bi}[1]{\begin{inspring}{#1}}
\newcommand{\ei}{\end{inspring}}

\newcommand{\beq}{\begin{equation}}
\newcommand{\eq}{\end{equation}}
\newcommand{\dC}{{\Bbb C}}

\usepackage[]{amsmath}
\usepackage{amssymb}
\usepackage{enumerate, amsfonts, amsthm, bbm}

\usepackage{amssymb,graphicx,tabularx}

\newcommand{\eqan}[1]{\begin{align} #1 \end{align}}

\newtheorem{theorem}{Theorem}

\newtheorem{lemma}[theorem]{Lemma}

\newtheorem{proposition}[theorem]{Proposition}

\begin{document}
\title{Novel heavy-traffic regimes for large-scale service systems}
\author{
A.J.E.M. Janssen \and
J.S.H. van Leeuwaarden \and
B.W.J. Mathijsen
\footnote{Eindhoven University of Technology, Department of Mathematics and Computer Science, P.O. Box 513, 5600 MB Eindhoven, The Netherlands.
\{a.j.e.m.janssen,j.s.h.v.leeuwaarden,b.w.j.mathijsen\}@tue.nl.}
}
        
\maketitle

\footnotetext[0]{This work was financially supported by The Netherlands Organization for Scientific Research (NWO) and by an ERC Starting Grant.}

\begin{abstract}
We introduce a family of heavy-traffic regimes for large scale service systems, presenting a range of scalings that include both moderate and extreme heavy traffic, as compared to classical heavy traffic. The heavy-traffic regimes can be translated into capacity sizing rules that lead to Economies-of-Scales, so that the system utilization approaches 100\% while congestion remains limited. 
We obtain heavy-traffic approximations for stationary performance measures in terms of
asymptotic expansions, using a non-standard saddle point method, tailored to the specific form of integral expressions for the performance measures, in combination with the heavy-traffic regimes. 
\end{abstract}



\section{Introduction}
Facing costs and large populations in need of service, service operations in for instance call centers, health care and digital communications face the challenge of matching customer demand with provider capacity. Timely access and responsiveness should be balanced with the costs of service capacity, and the central question is thus how to match capacity and demand. By taking into account the natural fluctuations of demand, stochastic models have proved instrumental in both quantifying and improving the operational performance of service systems. A celebrated capacity sizing rule for large service systems, which arises from stochastic intuition, prescribes that capacity should be such that it can deal with the natural fluctuations of demand. Say the demand per period is given by some random variable $A$, with mean $\mu_A$ and variance $\sigma^2_A$. For systems facing large demand one then sets the capacity according to the rule $s=\mu_A+\beta \sigma_A$, which consists of a part $\mu_A$ that is minimally required to deal with the incoming demand, and a part $\beta \sigma_A$ which is the additional capacity that should protect the system against stochastically predictable yet unforeseeable fluctuations. The additional capacity $\beta \sigma_A$ hence presents a {\it hedge against variability}, which is of the order of the natural fluctuations in demand $\sigma_A$, and which can be fine-tuned by setting the constant $\beta$.

Large-scale service operations are often amenable to pursue the dual goal of high system utilization and short delays. In order to achieve this, capacity sizing rules, or heavy-traffic regimes, typically scale systems such that the system utilization approaches 100\% while both the demand and capacity grow large, hence rendering effects of Economies-of-Scale. The capacity sizing rule described above often fulfills this condition, and we can best describe this in terms of a setting in which the demand per period is generated by many customers. Consider a service system with $n$ independent customers and let $X$ denote the generic random variable that describes the individual customer demand per period, with mean $\mu_X$ and variance $\sigma_X^2$. Denote the service capacity by $s_n$, so that the system utilization is given by $\rho_n=n\mu_X/s_n$, where we index with $n$ to express the dependence on the scale at which the system operates. The traditional capacity sizing rule would then be $s_n=n\mu_X+\beta \sigma_X \sqrt{n} $. The standard heavy-traffic paradigm, which builds on the Central Limit Theorem, then prescribes to consider a sequence of systems indexed by $n$ with associated loads $\rho_n$ such that
\beq\label{bb1}
\sqrt{n}(1-\rho_n)\rightarrow \gamma=\frac{\beta\sigma_X^2}{\mu_X}>0, \quad {\rm as} \ n\to\infty.
\eq
Starting from this classical setting, we introduce a novel family of heavy traffic regimes by considering a wide range of capacity sizing rules for systems serving many customers and that operate close to 100\% system utilization. This novel family is best described in terms of a parameter $\alpha$ for which we assume that
\beq\label{bb}
n^{\alpha}(1-\rho_n)\rightarrow \gamma, \quad {\rm as} \ n\to\infty,  \ \gamma> 0.
\eq
The parameter $\alpha\geq 0$ defines a whole range of possible scaling regimes, including the classic case $\alpha=1/2$. %The case $\alpha=0$ corresponds to {\it large deviations} (LD). The range $\alpha\in(0,1/2)$ is referred to as {\it moderate deviations} (MD), which covers all cases in between LD and CLT.
In terms of a capacity sizing rule for systems with many customers, the condition \eqref{bb} is tantamount to $s_n=n\mu_X+\beta \sigma_X n^{1-\alpha}$. Similar capacity sizing rules have been considered in \cite{Bassamboo2010,maman} for many-server systems with uncertain arrival rates. Hence, for $\alpha\in(0,1/2)$ the variability hedge is relatively large, so that the regime parameterized by $\alpha\in(0,1/2)$ can  be seen as {\it moderate} heavy traffic: heavy traffic conditions in which the full occupancy is reached more slowly, as a function of $n$, than for classical heavy traffic. For opposite reasons the range $\alpha\in(1/2,\infty)$ corresponds to {\it extreme} heavy traffic due to a relatively small variability hedge. Note that the case $\alpha=0$ does not lead to 100\% system utilization when $n\to\infty$.

In this paper we apply \eqref{bb} to a specific stochastic model, and show that Economies-of-Scale can be achieved for a large range of $\alpha$, although the nature of the benefits obtained by operating on large scale depends on the precise capacity sizing rule (hence the parameter $\alpha$). We quantify performance in terms of stationary  measures: The mean and variance of the congestion in the system, and the probability of an empty system. For these performance measures we derive heavy-traffic limits under the scalings \eqref{bb} that
 are relatively simple functions of only the first two moments of the demand per period. Such parsimonious expressions are useful for quantifying and improving system behavior. The heavy-traffic limits, however, provide also qualitative insight into the system behavior. Our asymptotic analysis shows that mean congestion is $O(n^\alpha)$, which implies
that delays experienced by the customers are negligible for all values of $\alpha\in [0,1)$, are roughly constant for $\alpha=1$, and grow without bound for $\alpha>1$.  We expect this qualitative behavior to be universal for a wide range of stochastic models to which the regime \eqref{bb} is applied.
We further show the existence of the following trichotomy as $n\to \infty$ under \eqref{bb}: For $\alpha \in (0,1/2)$ the empty-system probability converges to $1$, for $\alpha\in (1/2,1)$ it converges to $0$, while only for $\alpha=1/2$ there is a limiting value in $(0,1)$. Hence, as expected, the system performance deteriorates with $\alpha$, in a rather crude way for the empty-system probability, and in only a mild way for mean congestion levels. The regime \eqref{bb} thus presents a range of possible capacity sizing rules that all lead to Economies-of-Scale, and depending on what is the desired nature of performance for a particular service system, an appropriate $\alpha$ can be selected. From the quantitative perspective, our detailed asymptotic analysis leads to more precise asymptotic estimates for the performance measures in heavy traffic, which reveal the exact manner in which the mean congestion is influenced by $\alpha$ and $\gamma$.

To explore the family of heavy-traffic regimes \eqref{bb}, we choose as a vehicle a specific discrete stochastic model, in which we divide time into periods of equal length, and model the net input in period $k$ as the difference between the incoming demand $A_n(k)$ and a fixed capacity $s_n$. Assuming demand is generated by $n$ independent customers, the demand per period can be written as $A_n(k) = \sum_{i=1}^n X_{i}(k)$, and is assumed to be integer-valued. The rule in \eqref{bb} thus specifies how the mean and variance of $A_n(k)$ (and hence of $X_i(k)$), and simultaneously $s_n$, will all grow to infinity as functions of $n$. Although the scaling \eqref{bb} can be applied to other possibly more general models, our model strikes a good balance between applicability and mathematical tractability. The vast majority of the literature for service systems builds on stochastic models for individual customer arrivals and departures under Markovian assumptions. In particular, the classical birth-death processes that describe $M/M/s_n$ and $M/M/s_n+M$ systems are the main drivers, both for plain performance evaluation, and for exploring capacity sizing rules in heavy-traffic regimes, see \cite{halfinwhitt,Borst2004}. These birth-death processes are driven by {\it arrival rates}, leading to sample paths in which customers leave and depart one by one, while our model represents processes embedded at equidistant time points, driven by {\it arrival counts} in the periods between these time points. Albeit at a rougher scale, or at a higher level of aggregation, a practical modeling advantage is that the random variable $A_n(k)$ leaves room for interpretations that do not rely on the Poisson assumption. Let us mention some possible interpretations. The canonical framework for large data-handling systems considers a buffer that receives messages from $n$ independent and identical information sources. Source $i$ generates $X_{i}(k)$ data packets in slot $k$, so that in total $A_n(k) = \sum_{i=1}^n X_{i}(k)$ packets join the buffer in slot $k$. The buffer depletes through an output channel with a maximum transmission capacity of $s_n$ packets per time slot. As such our model can be viewed as a discrete version of the Anick-Mitra-Sondhi model \cite{Anick1982}. Many-sources scaling became popular as it is well suited to modern telecommunications networks, in which a switch may have hundreds of different input flows, but the situation in which $A_n(k) = \sum_{i=1}^n X_{i}(k)$ can in principle be applied to any service system in which demand can be regarded as coming from many different inputs. 

Our stochastic model can be viewed as a discrete time bulk service queue, in which work is served in bulks of size $s_n$. This model is one of the canonical models in queueing theory, having a wide range of applications in fields like digital communication, wireless networks, road traffic, reservation systems, healthcare and many more (see \cite{Bruneel1993} and \cite[Chap.~5]{johanthesis} for an overview). In road traffic, the basic model for congestion at an intersection, known as the fixed-cycle traffic-light queue \cite{Newell1960,vanLeeuwaarden2006}, is related to our discrete bulk service queue. Then $s_n$ represents the maximum number of delayed cars in front of a traffic light that can depart during one green period, while $A_k(n)$ is the number of newly arriving cars during a consecutive green and red period. An example from healthcare is panel sizing  \cite{Zacharias2014}. Say a general practitioner has a pool of $n$ clients (typically in the order of thousands \cite{Green2008}), all of which are potential patients, and together require $A_k(n)$ consults per day. Further assume that the practitioner can see a maximum number of $s_n$ patients per day. What is then an appropriate patient panel size $n$, which strikes a reasonable balance between accessing medical care in a timely manner and restricting the time that the practitioner sits idle? The panel size application is one of many examples of an appointment book, referring to some schedule of appointments for a fixed period, with capacity $s_n$ appointments per period and newly arriving appointments $A_k(n)$ per period. See \cite{Dai2014} for another recent example of an appointment book in a healthcare setting, again in terms of our bulk service queue, with $A_n(k)$ the new patients per day and $s_n$ the number of available beds. For all examples above, and many more, our new class of heavy-traffic scalings \eqref{bb} presents capacity sizing rules for which the expected performance can be quantified using the results in this paper. This will be helpful in dimensioning the systems (How much capacity is needed to achieve a certain target performance?) while exploiting Economies-of-Scale. For appointment books our model together with the capacity sizing rules \eqref{bb} are particularly relevant for {\it advanced access} \cite{Green2008}, a scheduling approach in healthcare designed to reduce delays by offering every patient a same-day appointment, regardless of the urgency of the problem. In that way, patients do not have to wait long for appointments, and practices do not waste capacity by holding appointments in anticipation of urgent.

Next to the freedom to model different situations, another advantage of our model is that it is mathematically tractable, in the sense that it can be subjected to powerful mathematical methods from complex and asymptotic analysis. In order to establish the heavy-traffic limits we start from Pollaczek's formula for the transform of the stationary queue length distribution in terms of a contour integral. From this famous transform representation, contour integrals for the empty-system probability and the mean and variance of the congestion immediately follow. Contour integrals are often amenable to asymptotic evaluation (see e.g.~\cite{cohen}), particular for obtaining classical heavy- traffic asymptotics. We also subject the contour integral representations to asymptotic evaluation, but not under classical heavy traffic scaling. This asymptotic analysis requires a {\it non-standard} saddle point method (see \cite{flajolet} for an historical account on the application of the saddle point method in mathematics), tailored to the specific form of the integral expressions that arise under the capacity sizing rule \eqref{bb}. The saddle point method in its standard form is typically suited for large deviations regimes, for instance to characterize rare event probabilities, and cannot be applied to asymptotically characterize other stationary measures like the mean or mass at zero. Indeed, for our model the saddle point converges to one (as $n\to\infty$), which is a singular point of the integrand, and renders the standard saddle point method useless. Our non-standard saddle point method, originally proposed by \cite{debruijn}, is made specifically to overcome this complication. This leads to asymptotic expansions for the  performance measures, of which the limiting forms correspond to the heavy-traffic limits, and pre-limit forms present refined approximations for pre-limit systems ($n<\infty$) in heavy traffic. Such refinements to heavy-traffic limits are commonly referred to as {\em corrected diffusion approximations} \cite{siegmund,blanchetglynn,asmussen}.

\noindent{\bf Further connections to the literature.}
In this paper we consider a bulk service queue, which serves as one of the key models for the performance analysis of service systems. Under the family of scalings \eqref{bb} we establish heavy-traffic approximations that give insight into the system behavior in high occupancy scenarios. We like to point out that the results in this paper are all formulated for the special case of demand generated by many sources, so that the demand $A_n(k)$ can be described in terms of the sum $\sum_{i=1}^n X_{i}(k)$, but the methodology we develop is applicable to more general models that make no specific assumptions on the random variable $A_n(k)$ except for its second moment to exists. What is important is that Pollaczek's formula is available, so that the saddle point method can be applied. We now discuss two classes of stochastic systems for which the heavy-traffic regime \eqref{bb1} has been studied extensively, and for which our new family of regimes \eqref{bb} is largely unexplored. We discuss these classes because, despite the Pollaczek formula not to hold, we believe the qualitative results that we reveal for our particular model should to a large extent carry over to these settings as well, presenting some interesting avenues for further research (see \S\ref{subsec62}).

The first class concerns so-called  {\it nearly-deterministic} systems \cite{nd1,nd2}, denoted by $G_n/G_n/1$ system, where $G_n$ stands for {\it cyclic thinning} of order $n$, indicating that some  point process is thinned to contain only every $n$th point. As $n\to\infty$, the $G_n/G_n/1$ systems approach the deterministic $D/D/1$ system. For $G_n/G_n/1$ systems, \cite{nd1} establishes stochastic-process limits, and \cite{nd2} derives heavy-traffic limits for stationary waiting times. In the framework of \cite{nd1,nd2}, our stochastic model corresponds to a $D/G_n/1$ queue, where the sequence of service times $(A_n(k))_{k\geq 1}$ follows from a cyclically thinned sequence of i.i.d.~random variables $(X_i(k))$. It follows from \cite[Theorem 3]{nd2} that the rescaled stationary waiting time process converges under \eqref{bb1} to a reflected Gaussian random walk. Hence, the performance measures of the nearly deterministic system, under \eqref{lind} and  \eqref{bb1}, should be well approximated by the performance measures of the reflected Gaussian random walk, giving rise to heavy-traffic approximations. This connection is discussed in detail in \S\ref{subsec3.2}. It seems likely that results similar as in this paper can be obtained for applying the scaling \eqref{bb} to the nearly-deterministic systems in \cite{nd1,nd2}, and because Pollaczek's formula also applies to this setting, the non-standard saddle point method developed in this paper can provide the appropriate methodology. 

The second class concerns multi-server systems, and in particular the many-server regime (not to be confused with many-sources regime). When we interpret $s_n$ as the number of servers, instead of capacity per time slot or order of thinning, the scaling \eqref{bb1} is similar to the many-server heavy-traffic regime called QED (Quality-and-Efficiency Driven) or Halfin-Whitt regime, first developed by Halfin and Whitt \cite{halfinwhitt} for the $M/M/s_n$ system. The QED regime \eqref{bb1} is in many situations a highly effective way of scaling, because the probability of delay converges to a non-degenerate limit away from both zero and one, and the mean delay is asymptotically negligible as the number if servers grows large. The QED regime \eqref{bb1} is naturally positioned in between the Quality-Driven (QD) regime and the Efficiency-Driven (ED) regime. In the QD regime, the load remains bounded away from 1, which corresponds to setting $\alpha=0$ in \eqref{bb}. Hence, the range $\alpha\in(0,1/2)$ bridges the gap between the QED regime and the QD regime. Likewise, the ED regime corresponds to setting $\alpha=1$ in \eqref{bb}, so that the range $\alpha\in(1/2,1]$ connects the QED regime and ED regime. For the birth-death process describing the $M/M/s_n$ system, Maman \cite{maman} introduced a scaling similar to \eqref{bb}, and called it the QED-$c$ regime, also bridging the ED and QD regimes. \cite[Thm 4.1]{maman} says that the expected waiting time under the scaling $s_n = n\mu_X+\beta\sigma_Xn^{1-\alpha}$ is of order $s_n^{1-\alpha}$, which is equivalent to the expected queue length being of order $n^\alpha$ by Little's law. We should stress though that we expect the mathematical techniques that are needed to establish heavy-traffic results could be entirely different than in this paper, because Pollaczek's formula does not apply to many-server settings. The specific model assumptions will determine to a large extent the appropriate methodology. Under Markovian assumptions leading to the $M/M/s_n$ system, product-form solutions are available for the stationary distribution. This makes it possible to describe performance measures like the mean congestion directly in terms of real integrals. Where the saddle point method is used for integrals in the complex plane, the Laplace method (see e.g.~\cite{flajolet}) is used for real integrals. Hence, for the asymptotic evaluation of the $M/M/s_n$ system under the scaling \eqref{bb}, the Laplace method seems an appropriate methodology, although again one needs to deal with possible singularities in the integrand. For $G/D/s_n$ systems, which assume deterministic service times, it has been shown in \cite{jelenkovic} that using a decomposition property the dynamics of this multi-server systems can be captured in terms of a single-server system. Hence, for these systems, Pollaczek's formula applies, and our saddle point method can most likely be applied to obtain heavy-traffic results in the regimes \eqref{bb}. Under more general conditions, for instance leading to a $G/G/s_n$ system, it is simply unclear at this stage how to obtain precise heavy-traffic approximations for \eqref{bb}, because a tractable description of the performance measures is not available.

\noindent{\bf Structure of the paper.}
In \S\ref{sec1} we present in detail the model and the family of heavy-traffic scalings. In \S\ref{spSec} we introduce the saddle point method. In \S\ref{sec3} we apply the saddle point method for the mean congestion level. Theorem \ref{mainthm} gives for all heavy-traffic scalings the limiting behavior in terms of an integral expression. As a consequence, we show in
 Proposition \ref{prop1} that there are two types of heavy-traffic behavior, depending on whether $\alpha\in(0,1/2)$ or $\alpha\geq 1/2$.
In \S\ref{subsec3.2} we discuss for the case $\alpha=1/2$ the connection with the Gaussian random walk and the Riemann zeta function.
In fact, we show that for all $\alpha\geq 1/2$ there exists a connection between the integral expression in Theorem \ref{mainthm} and the Riemann zeta function.
In \S\ref{more} we apply the saddle point method to obtain several more heavy-traffic results, including refined heavy-traffic approximations for the mean congestion level, and the leading heavy-traffic behaviors for the variance of the stationary congestion level and for the empty-system probability. Numerical examples are given in \S\ref{numm}. Appendix \ref{app} presents a new self-contained derivation of Pollaczek's formula for the transform of the stationary waiting time in the $D/G/1$ system, which forms the point of departure for our analysis. In \S\ref{numm}, however, we do confirm through numerical experiments that under \eqref{bb}, various multi-server systems behave similar to our discrete bulk service queue.

\section{Model description and heavy-traffic regimes}\label{sec1}
We thus consider a discrete stochastic model in which time is divided into periods of equal length. At the beginning of each period $k=1,2,3,...$ new demand $A_n(k)$ arrives to the system. The demands per period $A_n(1),A_n(2),...$ are assumed independent and equal in distribution to some non-negative integer-valued random variable $A_n$.
The system has a service capacity $s_n\in\mathbb{N}$ per period, so that the recursion
\beq
\label{lind}
Q_{k+1} = \max\{Q_k + A_n(k) - s_n,0\},\qquad k=1,2,...,
\eq
assuming $Q_0=0$, gives rise to a Markov chain $(Q_k)_{k\geq 1}$ that describes the congestion in the system over time. The probability generation function (pgf)
\beq \label{e2}
A_n(z)=\sum_{j=0}^{\infty} \mathbb{P}(A_n=j) z^j
\eq
is assumed analytic in a disk $|z|<r$ with $r>1$, which implies that all moments of $A_n$ exist. We also assume that
\beq \label{e3}
A_n'(1)=\mathbb{E}A_n(k)=\mu_A<s_n.
\eq

Under the assumption (\ref{e3}) the function $z^{s_n}-A(z)$ has exactly $s$ zeros in the closed unit disk, one of these being $z=1$ (see \cite{rouche}).
We further assume that $\mathbb{P}(A=j)>0$ for some $j>s_n$.
Under this assumption the function
$z^{s_n}-A(z)$ also has zeros outside $|z|\leq1$, and we let $r_0$ be the minimum modulus of these zeros.
The number $r_0$ is the unique zero of $z^{s_n}-A(z)$ with real $z>1$; see e.g.~\cite{ref11}.

%It may happen that there are several zeros of modulus $r_0$, but in all cases $r_0$ itself is a zero of $z^s-A(z)$. In case that $z^s-A(z)$ has no zeros outside $|z|\leq1$ we put $r_0=\infty$.

Under the assumption (\ref{e3}) the stationary distribution $\lim_{k\to\infty}\mathbb{P}(Q_k=j)=\mathbb{P}(Q=j)$, $j=0,1,\ldots$ exists, with the random variable $Q$ defined as having this stationary distribution.

 We let
\beq \label{e4}
Q(w)=\sum_{j=0}^{\infty}\mathbb{P}(Q=j)w^j
\eq
be the pgf of the stationary distribution. $Q(w)$ is analytic in $|w|<r_0$, and given by Pollaczek's formula (see e.g.~\cite{awp, cohen})
\beq \label{e5}
Q(w)=\exp\,\Bigl[\,\frac{1}{2\pi i}\,\il_{|z|=1+\eps}\,\ln\Bigl(\frac{w-z}{1-z}\Bigr)\,\frac{(z^{s_n}-A(z))'}{z^{s_n}-A(z)}\,dz\Bigr] ,
\eq
where $\eps>0$ is such that $|w|<1+\eps<r_0$. In (\ref{e5}), the principal value of $\ln(\frac{w-z}{1-z})$ is chosen, which is analytic in the whole complex $z$-plane, except for a branch cut consisting of the straight line segment from $w$ to 1. In  Appendix \ref{app} we present a short proof of Pollaczek's formula in the discrete-queue setting that we have here.

Using $\mathbb{P}(Q=0)=Q(0)$, $\mu_Q=Q'(1)$ and $\sigma_Q^2  =  Q''(1)+Q'(1)-(Q'(1))^2$,   it follows by straightforward manipulations that
\begin{align} \label{e6}
\mathbb{P}(Q=0)&=\exp\,\Bigl[\frac{1}{2\pi i}\,\il_{|z|=1+\eps}\,\ln\Bigl(\frac{z}{z-1}\Bigr)\frac{(z^{s_n}-A(z))'}{z^{s_n}-A(z)}\,dz\Bigr] , \\
\label{e7}
\mu_Q&=\frac{1}{2\pi i}\,\il_{|z|=1+\eps}\,\frac{1}{1-z}~\frac{(z^{s_n}-A(z))'}{z^{s_n}-A(z)}\,dz ,\\
\label{e8}
\sigma_Q^2 &= \frac{1}{2\pi i}\,\il_{|z|=1+\eps}\,\frac{-z}{(1-z)^2}~\frac{(z^{s_n}-A(z))'}{z^{s_n}-A(z)}\,dz .
\end{align}
Because $s_n$ appears directly in expressions \eqref{e6}-\eqref{e8}, we will be conducting our analysis with respect to $s_n$ rather than $n$. Note that this has no consequences for our results on the convergence speed of the performance metrics, since $s_n = O(n)$. Furthermore, we will omit the index $n$ when describing the capacity $s_n$ in the remainder of the paper for brevity.  

We next discuss in more detail the family of heavy-traffic scalings considered in this paper, which combines two features. First, we have assumed that
 $A_k^{n}$ is in distribution equal to the sum of work generated by all sources, $X_{1,k}+...+X_{n,k}$, where the $X_{i,k}$ are for all $i$ and $k$ i.i.d.~copies of a random variable $X$, of which the pgf $X(z)=\sum_{j=0}^{\infty}\mathbb{P}(X=j)z^j$ has radius of convergence $r>1$, and
\beq \label{e9}
0<\mu_A=n\mu_X=n X'(1)<s .
\eq
Hence
\beq \label{e10}
\vart:=\frac{n}{s}\in(0,1/\mu_X) .
\eq
Second, we scale the system according to \eqref{bb}, for which we assume that
\beq \label{e11}
\rho_s=\vart\,\mu_X=1-\frac{\gamma}{s^\alpha}
\eq
in which $\gamma>0$ is bounded away from 0 and $\infty$ as $s\pr\infty$.

The condition that $\mathbb{P}(A=j)>0$ for some $j>s$ holds when the degree $d$ of $X(z)$ (with $d=\infty$ if $X(z)$ is not a polynomial) is such that $nd>s$.

To avoid certain complications when applying the saddle point method, we further assume that
\beq \label{e12}
|X(z)|<X(r_1) ,~~~~~~|z|=r_1\,,~~z\neq r_1 ,
\eq
for any $r_1\in(0,r)$. This implies that $r_0$ is the unique zero of $z^s-A(z)$ on $|z|=r_0$.
%We refer to \cite{rouche}, Section~3, for how to proceed when this condition does not hold.
This condition is related to Cram\'er's condition, see \cite[pp.~189 and 355]{asmussen}, and it has also been used in \cite{relaxation}.
Condition \eqref{e12} holds when the set of all $j=0,1,\ldots$ such that $\mathbb{P}(X=j)>0$ is not contained in an  arithmetic progression with a ratio larger than one (see also \cite{rouche}).

\section{Non-standard saddle point method}\label{spSec}\ \
We illustrate our saddle point method for $\mu_Q$.
As a first step, we bring  (\ref{e7}) in a form which is amenable to saddle point analysis.
\begin{lemma}
\beq \label{e18}
\mu_Q  =  \frac{s}{2\pi i}\,\il_{|z|=1+\eps}\,\frac{g'(z)}{z-1}~\frac{\exp(s\,g(z))}{1-\exp(s\,g(z))}\,dz
\eq
with
\beq \label{e15}
g(z)={-}{\rm ln}\,z+\vart\,{\rm ln}(X(z)) .
\eq
\end{lemma}
\begin{proof}
With $A(z)=X^n(z)$,
\begin{align} \label{e13}
\frac{(z^s-A(z))'}{z^s-A(z)} & =  \frac{s\,z^{s-1}-n\,X'(z)\,X^{n-1}(z)}{z^s-X^n(z)} \nonumber \\
& =  \frac{s}{z}-\frac{s}{z}\,\Bigl(\frac{n}{s}~\frac{z\,X'(z)}{X(z)}-1\Bigr)\,\frac{z^{-s}\,X^n(z)}{1-z^{-s}\,X^n(z)} .
\end{align}
Write
$
z^{-s}\,X^n(z)=\exp(s g(z))$.
Noting that
\beq \label{e16}
\frac{1}{2\pi i}\,\il_{|z|=1+\eps}\,\frac{s}{z}~\frac{1}{1-z}\,dz=0 ,
\eq
and that
\beq \label{e17}
g'(z)=\frac1z\,\Bigl(\vart\,\frac{z\,X'(z)}{X(z)}-1\Bigr) ,
\eq
gives \eqref{e18}. \end{proof}

Let us now explain how the standard saddle point method can be applied to \eqref{e18}.
Since
\beq \label{e19}
g(1)=g(r_0)=0~;~~~~~~g(z)<0\,,~~1<z<r_0 ,
\eq
and by strict convexity of
\beq \label{e20}
z^{-s}\,X^n(z)=z^{-s}\,A(z)=\sum_{k=0}^{\infty}\,a_k\,z^{k-s} ,~~~~~~z\in(0,r) ,
\eq
 $g(z)$ has a unique minimum on $[1,r_0]$. This minimum is found by solving $z\in[1,r_0]$ from $g'(z)=0$, and this yields the equation
\beq \label{e21}
X(z)=\vart\,z\,X'(z) .
\eq
Denote the solution $z\in(1,r_0)$ of (\ref{e21}) by $z_{\rm sp}$, and observe that $z_{\rm sp}$ is a saddle point of $g(z)$, explaining the notation. Thus, the saddle point method can be used for the integral in (\ref{e18}) by taking $1+\eps=z_{\rm sp}$.

In the case that $\vart=n/s$ is bounded away from $1/\mu_X$ as $s\pr\infty$, we have that the minimum value of $g(z)$, $1\leq z\leq r_0$, is negative and bounded away from 0. Furthermore, $z_{\rm sp}$ is bounded away from 1, and the saddle point method can be applied in the classical way by replacing
\beq \label{e22}
\frac{\exp(s\,g(z))}{1-\exp(s\,g(z))}~~~~{\rm by}~~~~\exp(s\,g(z)) ,
\eq
at the expense of an exponentially small relative error, and performing an expansion of $g'(z)/(z_{\rm sp}-1)=d_1(z-z_{\rm sp})+O((z-z_{\rm sp})^2)$ with $d_1=g''(z_{\rm sp})/(z_{\rm sp}-1)\neq 0$.
 Using that $g(z^{\ast})=(g(z))^{\ast}$, where the $^*$ denotes complex conjugation, it can be shown that
\beq \label{e23}
\mu_Q=\frac{\exp(s\,g(z_{\rm sp}))}{(z_{\rm sp}-1)^2\,\sqrt{2\pi s\,g''(z_{\rm sp})}}\,(1+O(s^{-1})) .
\eq
%\todo{perhaps we need more explanation here about the standard case.}
%Full details for the derivation of a result of type (\ref{e23}) will be given in Section~\ref{sec3} in the nearly deterministic regime.

We next explain why the standard saddle point method does not work for the heavy-traffic scaling considered in this paper. Since we operate in (\ref{e11}),
 $\vart\,\mu_X\pr1$ as $s\pr\infty$, and
\begin{align} \label{e24}
z_{\rm sp}-1&=\frac{\gamma}{a_2\,s^\alpha}+O(s^{-2\alpha}) ,\\
 \label{e25}
g(z_{\rm sp})&=\frac{-\gamma^2}{2a_2s^{2\alpha}}+O(s^{-3\alpha}) ,\\
 \label{e26}
g''(z_{\rm sp})&=a_2+O(s^{-\alpha}) ,
\end{align}
where
\beq \label{e27}
a_2=\frac{\sigma_X^2}{\mu_X}-\frac{\gamma}{s^\alpha}\,\Bigl(\frac{\sigma_X^2}{\mu_X}-1\Bigr) .
\eq
Hence, $\exp(sg(z))$ near $z=z_{\rm sp}$ is (as $s\pr\infty$):
 vanishingly small when $\alpha\in(0,1/2)$,
 bounded away from 1, but non-negligible when $\alpha=1/2$,
and tending to 1 when $\alpha\in(1/2,\infty)$.
Furthermore, $(z-1)^{-1}$ in \eqref{e18} is unbounded near $z=z_{\rm sp}$ as $s\pr\infty$. Therefore, an adaptation of the standard saddle point method is required, and the resulting asymptotic form of $\mu_Q$ will deviate significantly from the standard case (\ref{e23}). In particular, since $z_{\rm sp}\pr1$, this asymptotic form will contain information from $X(z)$ at $z=1$, rather than at a point away from 1 as is the case in (\ref{e23}).

An appropriate adaptation of the saddle point method is described in \cite[\S~5.12]{debruijn} for the integral
\beq \label{e28}
\il_{-\infty}^{\infty}\,\frac{e^v}{v^2+s^{-\kappa}}\,e^{-sv^2}\,dv ,
\eq
where $s\pr\infty$ and $\kappa$ is a positive parameter. We shall use and further develop this adaptation in \S~\ref{sec3} for the integral in (\ref{e18}) for which it turns out that one should take the value of the parameter $\kappa$ equal to $2 \alpha$ (see below (\ref{e43a})). However, first we transform the integral (\ref{e18}) further so as to bring it in a form in which, as in (\ref{e28}), a true Gaussian, rather than the factor $\exp(s\,g(z))$, appears.

We use a substitution $z=z(v)$ in (\ref{e18}) with real $v$ and $z(0)=z_{\rm sp}$ such that for sufficiently small $v$,
\beq \label{e29}
g(z(v))=g(z_{\rm sp})-\tfrac12\,v^2\,g''(z_{\rm sp}) .
\eq
This is feasible, since
\beq \label{e30}
g(z)=g(z_{\rm sp})+\tfrac12\,g''(z_{\rm sp})(z-z_{\rm sp})^2\Bigl(1+\frac{g'''(z_{\rm sp})}{3g''(z_{\rm sp})}\,(z-z_{\rm sp})+...\Bigr) 
\eq
with $g''(z_{\rm sp})$ positive and bounded away from 0 as $s\pr\infty$. Hence, $z(v)$ can be found for small $v$ by inverting the equation
\beq \label{e31}
(z-z_{\rm sp})\Bigl(1+\frac{g'''(z_{\rm sp})}{3g''(z_{\rm sp})}\,(z-z_{\rm sp})+...\Bigr)^{1/2}=iv .
\eq
By Lagrange's inversion theorem \cite{debruijn}, there is a $\delta>0$ (independent of $s$) such that
\beq \label{e32}
z(v)=z_{\rm sp}+iv+\sum_{k=2}^{\infty}\,c_k(iv)^k ,~~~~~~|v|<\delta ,
\eq
with real coefficients $c_k$ (since $g(z)$ is real for real $z$) and
\beq \label{e33}
c_2={-}\,\frac{g'''(z_{\rm sp})}{6g''(z_{\rm sp})} .
\eq
Thus
\beq \label{e34}
z(v)=z_{\rm sp}+iv-c_2\,v^2+O(v^3) ,~~~~~~|v|\leq\tfrac12\,\delta ,
\eq
where the order term holds uniformly in $s$. The uniformity statement follows from an inspection of the usual argument
by which Lagrange's theorem is proved, noting that the inversion in \eqref{e29} with $g$ as in \eqref{e15} is considered for $\vart\to 1/\mu_X$, $z_{\rm sp}\to 1$ with radius
of convergence $r$ away from $1$. 

By (\ref{e12}) we can restrict the integration in (\ref{e18}) to a fixed but arbitrarily small subset of $|z|=z_{\rm sp}$ near $z=z_{\rm sp}$, at the expense of an exponentially small error. Furthermore, by Cauchy's theorem and again at the expense of an exponentially small error, the integration path can be deformed in accordance with the transformation in (\ref{e29})--(\ref{e34}). Set
\beq \label{e35}
q(v)=g(z_{\rm sp})-\tfrac12\,v^2\,g''(z_{\rm sp})
\eq
and note that from \eqref{e29}
\beq \label{e36}
g'(z(v))\,z'(v)={-}v\,g''(z_{\rm sp}) .
\eq
Then substituting $z=z(v)$ in (\ref{e18}),  $\mu_Q$ is given with exponentially small error by
\begin{eqnarray} \label{e37a}
\frac{s}{2\pi i}\,\il_{-\frac12\delta}^{\frac12\delta}\,\frac{g'(z(v))}{z(v)-1}~\frac{\exp(s\,g(z(v)))}{1-\exp(s\,g(z(v)))}z'(v)\,dv,% \nonumber \\[3.5mm]
%& & =~\frac{-s}{2\pi i}\,g''(z_{\rm sp})\,\il_{-\frac12\delta}^{\frac12\delta}\,\frac{v}{z(v)-1}~\frac{\exp(s\,q(v))}{1-\exp(s\,q(v))}\,dv .
\end{eqnarray}
which gives the following result.
\begin{lemma} \label{lemma2} The mean stationary congestion level is given with exponentially small error by
\beq \label{e37}
\mu_Q =~\frac{-s}{2\pi i}\,g''(z_{\rm sp})\,\il_{-\frac12\delta}^{\frac12\delta}\,\frac{v}{z(v)-1}~\frac{\exp(s\,q(v))}{1-\exp(s\,q(v))}\,dv .
\eq
\end{lemma}

%\todo{perhaps also turn thus into a lemma and move to later section}
In a similar fashion we get that $\mathbb{P}(Q=0)$ and $\sigma_Q^2$, see (\ref{e6}) and (\ref{e8}), are given, both with exponentially small error, by
\beq \label{e39}
\frac{-s}{2\pi i}\,g''(z_{\rm sp})\,\il_{-\frac12\delta}^{\frac12\delta}\,v\,{\rm ln}\Bigl(\frac{z(v)}{z(v)-1}\Bigr)\frac{\exp(s\,q(v))}{1-\exp(s\,q(v))}\,dv
\eq
and
\beq \label{e38}
\frac{-s}{2\pi i}\,g''(z_{\rm sp})\,\il_{-\frac12\delta}^{\frac12\delta}\,\frac{v\,z(v)}{(z(v)-1)^2}~\frac{\exp(s\,q(v))}{1-\exp(s\,q(v))}\,dv,
\eq
respectively.

\section{Heavy-traffic limits for the mean congestion level} \label{sec3}
In this section we apply the non-standard saddle point method explained in \S\ref{sec1} to the Pollaczek integral representation for the mean stationary congestion level $\mu_Q$. In \S\ref{subsec3.1} we first derive an integral representation for the leading order behavior of $\mu_Q$  with a relative error of order $O(s^{-1})$, which serves as a heavy-traffic approximation in the regime $\rho_s=1-\gamma/s^\alpha$ with $\alpha>0$. We also consider separately the cases of moderate heavy traffic ($\alpha\in(0,1/2)$) and extreme heavy traffic ($\alpha\in(1/2,\infty)$), for which the integral representation leads to vastly different alternative expressions. We find that $\mu_Q\pr 0$ more rapidly than any power of $1/s$ when $\alpha\in(0,1/2)$. When $\alpha\geq 1/2$ the saddle point method yields an integral representation with relative error $O(s^{-\min(1,\alpha)})$.
In \S\ref{subsec3.2} we specialize this general result to the CLT case $\alpha=1/2$, and make a connection with existing results.

\subsection{Leading order behavior in integral form} \label{subsec3.1}

 \begin{theorem}\label{mainthm}
The mean stationary congestion level is given by
\beq \label{e48a}
 \mu_Q=\frac{2}{\pi}\,\sigma_X\,\sqrt{\dfrac{s}{2\mu_X}}\,\il_0^{\infty}\,\frac{t^2}{d^2(s)+t^2}~\frac{\exp({-}d^2(s)-t^2)}{1-\exp({-}d^2(s)-t^2)}\,dt\,\left(1+O({s^{{-}\min(1,\alpha)}})\right)
\eq
with $
d^2(s) = s^{1-2\alpha}\gamma^2\mu_X/(2\sigma_X^2)$.
\end{theorem}

\begin{proof}
According to Lemma \ref{lemma2}, $\mu_Q$ is given with exponentially small error by (\ref{e37}) with $q(v)$ given in (\ref{e35}). Since $z({-}v)=z^{\ast}(v)$ for real $v$, we have
\begin{eqnarray} \label{e40}
\frac{v}{z(v)-1}+\frac{-v}{z({-}v)-1} &=& {-}2iv\,\frac{{\rm Im}(z(v))}{|z(v)-1|^2}\nonumber\\
&=&~\frac{-2iv^2+O(v^4)}{(z_{\rm sp}-1)^2+v^2-2c_2(z_{\rm sp}-1)\,v^2+O(v^4)} ,
\end{eqnarray}
where (\ref{e34}) and $c_k\in\mathbb{R}$ have been used. Using (\ref{e40}) in (\ref{e37}) and extending the integration range from $[{-}\tfrac12\delta,\tfrac12\,\delta]$ to $({-}\infty,\infty)$ while using symmetry of $q(v)$, we get that $\mu_Q$ is given with exponentially small error by
\begin{eqnarray} \label{e41}
\frac{s\,g''(z_{\rm sp})}{\pi}\,\il_0^{\infty}\,\frac{v^2+O(v^4)}{(z_{\rm sp}-1)^2+v^2-2c_2(z_{\rm sp}-1)\,v^2+O(v^4)}\frac{\exp(s\,q(v))}{1-\exp(s\,q(v))}\,dv .
\end{eqnarray}
With
\beq \label{e42}
B=\exp(s\,g(z_{\rm sp})) ,~~~~~~\eta=g''(z_{\rm sp}),
\eq
(\ref{e41}) takes the form
\begin{eqnarray} \label{e43}
\frac{s\eta}{\pi}\,\il_0^{\infty}\,\frac{v^2+O(v^4)}{(z_{\rm sp}-1)^2+v^2-2c_2(z_{\rm sp}-1)\,v^2+O(v^4)} \cdot \frac{B\,\exp({-}\tfrac12\,s\,\eta\,v^2)}{1-B\,\exp({-}\tfrac12\,s\,\eta\,v^2)}\,dv .
\end{eqnarray}
In leading order, the integrand in (\ref{e43}) has the form
\beq \label{e43a}
\frac{B\,v^2\,\exp(-s\,D\,v^2)}{(v^2+C\,s^{-2\alpha})(1-\exp({-}s\,D\,v^2))},
\eq
and this is reminiscent of the integrand in (\ref{e28}) for the case  $\kappa=2\alpha$. Hence, it is natural to proceed as in \cite[\S 5.12]{debruijn}. The substitution $v=t\sqrt{{2}/(s\eta)}$ brings (\ref{e43}) into the form
\begin{eqnarray} \label{e44}
\frac{2}{\pi}\sqrt{\tfrac12 s\,\eta}\il_0^{\infty}\frac{t^2(1+O(t^2/s))}{\tfrac12 s\,\eta(z_{\rm sp}-1)^2+t^2-2c_2(z_{\rm sp}-1)t^2+O(t^4/s)} \,\frac{B\,\exp({-}t^2)}{1-B\,\exp({-}t^2)}dt .
\end{eqnarray}
From (\ref{e24})--(\ref{e27}) and (\ref{e42}),
\begin{align}
\frac{2}{\pi}\,\sqrt{\frac{s\eta}{2}} &= \frac{2}{\pi}\,\sigma_X\,\sqrt{\frac{s}{2\,\mu_X}}\,(1+O(s^{-\alpha})),\label{y45}\\
\tfrac12\,s\,\eta\,(z_{\rm sp}-1)^2 &= d^2(s) + O(s^{1-3\alpha}),\label{y46}\\
2\,c_2(z_{\rm sp}-1) &= O(s^{-\alpha}),\label{y47}\\
s\,g(z_{\rm sp}) &= -d^2(s) + O(s^{1-3\alpha}),\label{y48}
\end{align}
where
\beq \label{y49}
d^2(s) = \frac{b_0^2}{s^{2\alpha-1}},\quad b_0^2 := \frac{\gamma^2\mu_X}{2\,\sigma_X^2}.
\eq
In the case that $2\alpha-1<0$, we have that $\tfrac12\,s\,\eta\,(z_{\rm sp}-1)^2 \to \infty$ and that
\beq \label{y50}
B = \exp(s\,g(z_{\rm sp})) = O(\exp({-}b^2\,s^{1-2\alpha}))
\eq
for any $b\in(0,b_0)$. From \eqref{e44} it follows then that $\mu_Q = O(\exp({-}b^2\,s^{1-2\alpha}))$ for any $b\in(0,b_0)$.
In the case that $2\,\alpha-1\geq 0$, we have that $d^2(s)$ is bounded, and using that $1/s^{3\alpha-1} = O(d^2(s)/s^\alpha)$, we get
\begin{align}
\tfrac12\,s\,\eta\,(z_{\rm sp}-1)^2 &+ t^2-2\,c_2\,(z_{\rm sp}-1)\,t^2+O(t^4/s) \nonumber\\
&= d^2(s) + t^2 + O\left(s^{-\alpha}\,(d^2(s)+t^2)\right) + O(t^4/s)\nonumber\\
&= \left(d^2(s)+t^2\right)\left(1+O(s^{-\alpha})+O(t^2/s)\right).\label{y51}
\end{align}
Hence, in this case,
\beq \label{y52}
\frac{t^2(1+O(t^2/s))}{\tfrac12 s\,\eta(z_{\rm sp}-1)^2+t^2-2c_2(z_{\rm sp}-1)t^2+O(t^4/s)}
 = \frac{t^2}{d^2(s)+t^2}\left(1+O(s^{-\alpha})+O(t^2/s)\right),
\eq
where we restrict to $t$ in a range $[0,s^{1/4}]$. Furthermore,
\begin{align}
1-B\,\exp(-t^2) &= 1-\exp({-}d^2(s)-t^2)\,\left(1+d^2(s)\,O(s^{-\alpha})\right)\nonumber\\
&= (1-\exp({-}d^2(s)-t^2))\,\Bigl(1+\frac{d^2(s)\,\exp(-d^2(s)-t^2)}{1-\exp({-}d^2(s)-t^2)}O(s^{-\alpha})\Bigr)\nonumber\\
&=(1-\exp({-}d^2(s)-t^2))\,\Bigl(1+\frac{d^2(s)}{\exp(d^2(s)+t^2)-1}O(s^{-\alpha})\Bigr)\nonumber\\
&=(1-\exp({-}d^2(s)-t^2))\,\Bigl(1+\frac{d^2(s)}{d^2(s)+t^2}O(s^{-\alpha})\Bigr)\nonumber\\
&= (1-\exp({-}d^2(s)-t^2))\,(1+O(s^{-\alpha})),\label{y53}
\end{align}
%since for all $t\geq 0$
%\beq \label{y54}
%0\leq \frac{d^2(s)}{1-\exp({-}d^2(s)-t^2)}\leq \frac{d^2(s)}{1-\exp({-}%d^2(s))}\leq 1+d^2(s),
%\eq
%where it has been used that
%\beq \label{y55}
%\frac{x}{1-e^{-x}} \leq 1+x,\quad x\geq 0.
%\eq
where it has been used that
\beq \label{y55a}
e^x-1\geq x,\quad x\geq 0.
\eq
It follows therefore that
\beq \label{y56}
\frac{B\,\exp({-}t^2)}{1-B\,\exp({-}t^2)} = \frac{\exp({-}d^2(s)-t^2)}{1-\exp({-}d^2(s)-t^2)}\,(1+O(s^{-\alpha})).
\eq
Combining the three items \eqref{y45}, \eqref{y52} and \eqref{y56}, we obtain for \eqref{e44} the result
%\begin{align}
%\frac{2}{\pi}\,\sigma_X\,\sqrt{\frac{s}{2\,\mu_X}}(1+O(s^{-\alpha}))\,\il_0^{\infty}\frac{t^2}{d^2(s)+t^2}& \,\left(1+O(s^{-\alpha})+O(t^2/s)\right)\nonumber \\
%& \cdot~\frac{\exp({-}d^2(s)-t^2)}{1-\exp({-}d^2(s)-t^2)}\,(1+O(s^{-\alpha}))\,dt\label{y57}
%\end{align}
\beq \label{y58}
\frac{2}{\pi}\,\sigma_X\,\sqrt{\frac{s}{2\,\mu_X}}  \il_0^{\infty}\frac{t^2}{d^2(s)+t^2} \cdot \frac{\exp({-}d^2(s)-t^2)}{1-\exp({-}d^2(s)-t^2)}dt
\left(1+O(s^{-\alpha})+O(s^{-1})\right),
\eq
where the integration range $[0,\infty)$ is, at the expense of relative errors of type\\ $\exp({-}s^{1/4})$, first restricted to the range $[0,s^{1/4}]$, where \eqref{y52} holds, and then restored again to the full range. \end{proof}

Theorem \ref{mainthm} gives the leading-order behavior of $\mu_Q$ as $s\pr\infty$ with a relative error of $O(s^{{-}\min(1,\alpha)})$. By considering in more detail the integral expressions, we obtain the following result, describing two different heavy-traffic behaviors.

\begin{proposition}\label{prop1}
If $\alpha\in(0,1/2)$ the mean congestion level satisfies
\beq \label{y59}
\mu_Q=O\left(\exp(-b^2s^{1-2\alpha})\right),
\eq
for any $b\in (0,b_0)$. If $\alpha\in[1/2,\infty)$ the mean congestion level $\mu_Q$ is given by
\beq \label{y60}
s^\alpha\,\frac{\sigma_X^2}{2\mu_X\gamma}\,\left(1+O(s^{\max(1/2-\alpha,-1)})\right).
\eq
\end{proposition}
The first assertion in Proposition \ref{prop1} follows from the observation in \eqref{y50}, together with \eqref{e44}. The second assertion is based on a connection between the integral in Theorem \ref{mainthm} and the Riemann zeta function, which is explained in the next subsection.

\subsection{Classical heavy traffic and the Gaussian random walk}
\label{subsec3.2}
We now build on Theorem \ref{mainthm} to obtain further results for the  classical heavy traffic case $\alpha=1/2$,
for which we know from \cite[Theorem 3]{nd2} that the rescaled congestion process converges under \eqref{bb1} to a reflected Gaussian random walk. The latter is defined as
 $(S_\beta(k))_{k\geq 0}$ with $S_\beta(0)=0$ and
\begin{equation}
S_\beta(k)=Y_1+\ldots+Y_k
\end{equation}
with $Y_1,Y_2,\ldots$ i.i.d.~copies of a normal random variable with mean $-\beta$ and variance 1.
Assume $\beta>0$ (negative drift), and denote the all-time maximum of this random walk by ${M}_\beta$.

 Denote by $Q^{(s)}_\infty$ the stationary congestion level for a fixed $s$ (that arises from taking
 $k\to \infty$ in \eqref{lind}), and remember that we have assumed $\vartheta=n/s$ fixed.
Then, using $\rho_s=1-\gamma/\sqrt{s}$, with
 \begin{equation}\label{gammachoice}
 \gamma=\frac{\beta\sigma_X}{\mu_X\sqrt{\vartheta}},
\end{equation}
the spatially-scaled stationary congestion levels reach the limit
$Q^{(s)}_\infty/(\sigma_X\sqrt{n}) \stackrel{d}{\to} {M}_\beta$ as $s,n\to\infty$ (see \cite{jelenkovic,nd1,nd2}). From \cite[Theorem 4]{nd2} we then know that under the standard heavy-traffic scaling \eqref{bb1}
   \eqan{
   \frac{\mathbb{E}Q^{(s)}_\infty}{\sigma_X\sqrt{n}}\to \mathbb{E}{M}_\beta, \quad {\rm as} \ s,n\to\infty,
   }
   from which it follows that
   \beq \label{e48}
\mu_Q\approx \sigma_X\sqrt{n}\ \mathbb{E}M_\beta.
\eq
The random variable ${M}_\beta$ was studied in  \cite{changperes,jllerch}. In particular, \cite[Thm.~2]{jllerch} yields, for $\beta<2\sqrt{\pi}$,
\begin{eqnarray}\label{wdfegfw571}
\mathbb{E}{M}_\beta= \frac{1}{2\beta}+\frac{\zeta(1/2)}{\sqrt{2\pi}}+\frac{\beta}{4}+\frac{\beta^2}{\sqrt{2\pi}}\sum_{r=0}^{\infty}\frac{\zeta(-1/2-r)}{r!(2r+1)(2r+2)}\left(\frac{-\beta^2}{2 }\right)^r,
\end{eqnarray}
and hence, for small values of $\beta$,
   \beq \label{estimate}
\mu_Q\approx \sigma_X\sqrt{n}\ \mathbb{E}M_\beta \approx \frac{\sigma_X\sqrt{n}}{2\beta} = \sqrt{s}\,\frac{\sigma_X^2}{2\mu_X\gamma}.
\eq
We will now show how the approximation \eqref{estimate} follows from Theorem \ref{mainthm}, and also how similar steps give rise to Proposition \ref{prop1}.

Consider the integral
\beq \label{e49}
G_0(b)=G_1(b)-G_2(b)=\il_0^{\infty}\,\frac{t^2}{b^2+t^2}~\frac{\exp({-}b^2-t^2)}{1-\exp({-}b^2-t^2)}\,dt ,
\eq
where $b>0$ and
\beq \label{e50}
G_1(b)=\il_0^{\infty}\,\frac{\exp({-}b^2-t^2)}{1-\exp({-}b^2-t^2)}\,dt\,,~~~~G_2(b)=\il_0^{\infty}\,\frac{b^2}{b^2+t^2}~\frac{\exp({-}b^2-t^2)}{1-\exp({-}b^2-t^2)}\,dt .
\eq
We have, as in \cite[\S 2]{jllerch},
\begin{align} \label{e51}
G_1(b) & =  \sum_{k=0}^{\infty}\:\il_0^{\infty}\,\exp({-}(k+1)(b^2+t^2))\,dt \nonumber \\
& =  \frac{\sqrt{\pi}}{2}\,\sum_{k=0}^{\infty}\,\frac{e^{-(k+1)b^2}}{\sqrt{k+1}} = \frac{\sqrt{\pi}}{2}\,e^{-b^2}\,\Phi(e^{-b^2},1/2,1) \nonumber \\
& =  \frac{\pi}{2b}+\frac{\sqrt{\pi}}{2}\,\sum_{r=0}^{\infty}\,\zeta(\tfrac12-r)\,\frac{({-}1)^r\,b^{2r}}{r!} ,
\end{align}
where $\Phi(z,s,v)$ is Lerch's transcendent and where the last identity holds when $0<b<\sqrt{2\pi}$.

As to $G_2(b)$, we make a connection with the complementary error function
\beq \label{e52}
{\rm erfc}(z)=\frac{2}{\sqrt{\pi}}\,\il_z^{\infty}\,e^{-t^2}\,dt=\frac{2}{\pi}\,e^{-z^2}\,\il_0^{\infty}\,\frac{e^{-z^2t^2}}{1+t^2}\,dt ,
\eq
see \cite[Secs.~7.2 and 7.7.1]{ref5}. We thus compute
\begin{align} \label{e53}
G_2(b) & =  \sum_{k=0}^{\infty}\,e^{-(k+1)b^2}\,\il_0^{\infty}\,\frac{b^2}{b^2+t^2}\,e^{-(k+1)t^2}\,dt \nonumber \\
& =  \frac{\pi}{2}\,b\,\sum_{k=0}^{\infty}\,{\rm erfc}(b\,\sqrt{k+1}) .
\end{align}
From \cite[(4.3) and (4.23)]{jllerch},
\beq \label{e54}
\sum_{n=1}^{\infty}\,\frac{1}{\sqrt{2\pi}}\,\il_{\beta\sqrt{n}}^{\infty}\,e^{-x^2/2}\,dx= \frac{1}{2\beta^2}-\frac14-\frac{1}{\sqrt{2\pi}}\,\sum_{r=0}^{\infty} \frac{\zeta({-}1/2-r)({-}1/2)^r} {r!\,(2r+1)}\,\beta^{2r+1}
\eq
in which $0<\beta<2\sqrt{\pi}$. Taking $\beta=b\,\sqrt{2}$ in (\ref{e54}), we get
\beq \label{e55}
G_2(b)=\frac{\pi}{4b}-\frac{\pi}{4}\,b-\sqrt{\pi}\,\sum_{r=0}^{\infty}\,\frac{\zeta({-}1/2-r)({-}1)^r\,b^{2r+2}}{r!\,(2r+1)}
\eq
when $0<b<\sqrt{2\pi}$. The two results in (\ref{e51}) and (\ref{e55}) can be combined, as in \cite[end of \S\ref{sec4}]{jllerch}, and this yields
\beq \label{e56}
G_0(b)=\frac{\pi}{4b}+\frac{\pi}{4}\,b+\frac{\sqrt{\pi}}{2}\,\zeta(1/2)+\sqrt{\pi}\,\sum_{r=0}^{\infty}\frac{\zeta({-}1/2-r)({-}1)^r\,b^{2r+2}}{r!\,(2r+1)(2r+2)}
\eq
when $0<b<\sqrt{2\pi}$. Note that in \cite[\S 6]{jllerch} and \cite[Theorem~2]{cumulants} alternative infinite series expressions are given for the series in (\ref{e56}) that converge for all $b>0$.

Using (\ref{e56}) in (\ref{e48}), we find that the leading order behavior of $\mu_Q$ is given as
\beq \label{e57}
\sigma_X\,\sqrt{\dfrac{s}{2\mu_X}}\,\left[\frac{1}{2b_0}+\frac{b_0}{2}+\frac{\zeta(1/2)}{\sqrt{\pi}}+\frac{2}{\sqrt{\pi}}\,\sum_{r=0}^{\infty}\,\frac{\zeta({-}1/2-r)({-}1)^r b_0^{2r+2}} {r!\,(2r+1)(2r+2)}\right]
\eq
with relative error of $O(s^{-1/2})$ in which $b_0$ is given by \eqref{y49}. The expression (\ref{e57}) is exactly equal to the right-hand side of \cite[(4.25)]{jllerch} times $\sqrt{s}$ when we take there $\sigma=\mu=1$ and $\beta=b_0\,\sqrt{2}$.
Notice that, with $\gamma$ as in \eqref{gammachoice},
\beq \label{e57a}
\sigma_X\,\sqrt{\dfrac{s}{2\mu_X}}\frac{1}{2b_0}=\frac{\sigma_X\sqrt{n}}{2\beta},
\eq
which confirms the approximation \eqref{estimate}.

According to Theorem \ref{mainthm}, we have for $\alpha\geq 1/2$,
\beq \label{y61}
\mu_Q = \frac{2}{\pi}\,\sigma_X\,\sqrt{\frac{s}{2\,\mu_X}}G_0(d(s))\,\left(1+O(s^{{-}\min(1,\alpha)})\right).
\eq
When $\alpha=1/2$, so that $d(s) = b_0$ is independent of $s$, the series representation for $G_0$ in \eqref{e56} can be used, as long as $b_0\in(0,\sqrt{2\pi})$. When $\alpha>1/2$, we have that $d(s) = b_0/s^{\alpha-1/2}\to 0$ as $s\to\infty$, and so this series representation can be used when $s$ is large enough. We then have from \eqref{e56} and $b_0^2 = \gamma^2\mu_X/2\,\sigma_X^2$, while replacing the whole series at the right-hand side by $O(b^2)$, for $\mu_Q$ the leading order behavior
\beq \label{y62}
s^\alpha\left[\frac{\sigma_X^2}{2\,\gamma\,\mu_X}+\frac{\sigma_X\,\zeta(1/2)}{\sqrt{2\,\pi\,\mu_X}}\,\frac{1}{s^{\alpha-1/2}}+\frac{1}{4}\,\gamma\,\frac{1}{s^{2\alpha-1}}+O(s^{3/2-3\alpha})\right]
\eq
with relative error $O(s^{{-}\min(1,\alpha)})$. Retaining the constant term $\sigma_X^2/(2\gamma\mu_X)$ and estimating the other terms between the brackets in \eqref{y62} as $O(s^{1/2-\alpha})$, we get Proposition \ref{prop1}.
\section{More heavy-traffic results}\label{more}
In this section we apply the non-standard saddle point method to obtain several more heavy-traffic results. In \S\ref{subsec3.3} we derive refined heavy-traffic approximations for the mean congestion level by considering higher-order correction terms. In \S\ref{sec4} we derive the leading heavy-traffic behavior for the variance of the stationary congestion level, and in \S\ref{sec5} for the empty-system probability. To keep the developments tractable, we restrict \S\ref{subsec3.3} to $\alpha=1/2$, and \S\ref{sec4} and \S\ref{sec5} to $\alpha\in(0,1]$, although the same technique will work for all values $\alpha>0$.

\subsection{Correction term for the mean congestion level in the case $\alpha = 1/2$} \label{subsec3.3}
Our saddle point method not only establishes the leading-order heavy-traffic approximations, but also allows to derive refinements to these approximations. In this section we demonstrate how this works for the mean congestion level in the case $\alpha=1/2$.

To obtain a refinement or correction term from (\ref{e44}), we must be more precise about the $O(s^{{-}\alpha})$ terms that occur in the approximations in \S\ref{subsec3.1} for $\frac12\,s\,\eta(z_{\rm sp}-1)^2$, $B$ and $\sqrt{s\,\eta/2}$. When higher-order corrections are required, we should include higher-order terms in the approximations of these quantities, and be more specific about the $O(t^2/s)$ and $O(t^4/s)$  in the integrand in (\ref{e44}).

Denote, see \eqref{e10} and (\ref{e15}) with $\vart=(1-\gamma/s^\alpha)\,\mu_X^{-1}$,
\beq \label{e58}
a_i=g^{(i)}(1);~~~~~~g(z)={-}{\rm ln}\,z+\vart\,{\rm ln}\,X(z) .
\eq
Dropping the $X$ from $\mu_X$ and $\sigma_X^2$ for brevity, we have
\beq \label{e59}
a_1={-}\,\frac{\gamma}{s^\alpha} ,~~~~~~a_2=\frac{\sigma^2}{\mu}-\frac{\gamma}{s^\alpha}\,\Bigl(\frac{\sigma^2}{\mu}-1\Bigr) ,
\eq
\beq \label{e60}
a_3={-}2+\Bigl(1-\frac{\gamma}{s^\alpha}\Bigr)\Bigl(\frac{X'''(1)}{X'(1)}-3X''(1)+2(X'(1))^2\Bigr) .
\eq
For the purpose of finding a first-order correction term, we note that
\begin{align} \label{e61}
\eta&=g''(z_{\rm sp})=a_2+(z_{\rm sp}-1)\,a_3+O(s^{-1}) ,\\
\label{e62}
z_{\rm sp}-1&={-}\,\frac{a_1}{a_2}-\frac{a_3}{2a_2}\,\Bigl(\frac{a_1}{a_2}\Bigr)^2+O(s^{-3/2}) ,\\
 \label{e63}
c_2&={-}\,\frac{g'''(z_{\rm sp})}{6g''(z_{\rm sp})}={-}\,\frac{a_3}{6a_2}+O(s^{-1/2}) ,\\
 \label{e64}
g(z_{\rm sp})&={-}\,\frac{a_1^2}{2a_2}-\frac{a_3}{6a_2^3}\,a_1^3+O(s^{-2}) .
\end{align}
This gives rise to
\begin{align} \label{e65}
\sqrt{\tfrac12\,s\,\eta}&=\sigma\,\sqrt{\dfrac{s}{2\mu}}\,\Bigl(1+\frac{C_1}{\sqrt{s}}+O(s^{-1})\Bigr) ,\\
\tfrac12\,s\,\eta(z_{\rm sp}-1)^2&=\frac{\gamma^2\,\mu}{2\sigma^2}+\frac{C_2}{\sqrt{s}}+O(s^{-1}) ,\\
\label{e67}
2c_2(z_{\rm sp}-1)&=\frac{C_3}{\sqrt{s}}+O(s^{-1}) ,\\
 \label{e68}
B=\exp(s\,g(z_{\rm sp}))&=\exp\Bigl({-}\,\frac{\gamma^2\,\mu}{2\sigma^2}\Bigr)\Bigl(1+\frac{C_4}{\sqrt{s}}+O(s^{-1})\Bigr) ,
\end{align}
with explicitly computable constants $C_1$, $C_2$, $C_3$, $C_4$. Remembering that $b_0^2=\gamma^2\mu/2\sigma^2$, see \eqref{y49}, we then get with errors of order $1/s$
\begin{eqnarray} \label{e69}
& \mbox{} & \frac{t^2(1+O(t^2/s))}{\frac12\,s\,\eta(z_{\rm sp}-1)^2+t^2-2c_2(z_{\rm sp}-1)\,t^2+O(t^4/s)} \nonumber \\[3mm]
& & =~\frac{t^2}{b_0^2+t^2}-\frac{1}{\sqrt{s}}\,\Bigl((C_2+ b_0^2\,C_3)\,\frac{t^2}{(b_0^2+t^2)^2}-C_3\,\frac{t^2}{b_0^2+t^2}\Bigr) ,
\end{eqnarray}
and
\beq \label{e70}
\frac{B\,\exp({-}t^2)}{1-B\,\exp({-}t^2)}=\frac{\exp({-}b_0^2-t^2)}{1-\exp({-}b_0^2-t^2)}+\frac{C_4}{\sqrt{s}}~\frac{\exp({-}b_0^2-t^2)}{(1-\exp({-}b_0^2-t^2))^2} .
\eq
Using (\ref{e65}), (\ref{e69}) and (\ref{e70}) in (\ref{e44}) we get with an absolute error of order $1/\sqrt{s}$
\begin{eqnarray} \label{e71}
\mu_Q & = & \frac{2}{\pi}\,\sigma\,\sqrt{\dfrac{s}{2\mu}}\,\Bigl(1+\frac{C_1}{\sqrt{s}}\Bigr) \il_0^{\infty}\,\Bigl(\frac{t^2}{b_0^2+t^2}-\frac{1}{\sqrt{s}}\,\Bigl((C_2+b_0^2\,C_3)\,\frac{t^2}{(b_0^2+t^2)^2}-C_3\,\frac{t^2}{b_0^2+t^2}\Bigr)\Bigr) \nonumber \\[3.5mm]
& & \cdot~\Bigl(\frac{\exp({-}b_0^2-t^2)}{1-\exp({-}b_0^2-t^2)}+\frac{C_4}{\sqrt{s}}~\frac{\exp({-}b_0^2-t^2)}{(1-\exp({-}b_0^2-t^2))^2}\Bigr)\,dt \nonumber \\[3.5mm]
& = & \frac{2\sigma}{\pi}\,\sqrt{\dfrac{s}{2\mu}}\,G_0(b_0)\nonumber\\
& &+ ~\frac{2\sigma}{\pi}\,\sqrt{\dfrac{1}{2\mu}}\,((C_1+C_3)\,G_0(b_0)-(C_2+b_0^2\,C_3)\,G_3(b_0)+C_4\,G_4(b_0)) ,
\end{eqnarray}
where $G_0$ is as in (\ref{e49}), and
\begin{align} \label{e72}
G_3(b_0)&=\il_0^{\infty}\,\frac{t^2}{(b_0^2+t^2)^2}~\frac{\exp({-}b_0^2-t^2)}{1-\exp({-}b_0^2-t^2)}\,dt ,\\
\label{e73}
G_4(b_0)&=\il_0^{\infty}\,\frac{t^2}{b_0^2+t^2}~\frac{\exp({-}b_0^2-t^2)}{(1-\exp({-}b_0^2-t^2))^2}\,dt .
\end{align}
We shall express the integrals in (\ref{e72}) and (\ref{e73}) in terms of $\zeta$-functions. By partial integration
\begin{align} \label{e74}
G_3(b) & =  \frac12\,\il_0^{\infty}\,\frac{1}{b^2+t^2}~\frac{\exp({-}b^2-t^2)}{1-\exp({-}b_0^2-t^2)}\,dt -~\il_0^{\infty}\,\frac{t^2}{b^2+t^2}~\frac{\exp({-}b^2-t^2)}{(1-\exp({-}b^2-t^2))^2}\,dt \nonumber \\
& =  \frac{1}{2b^2}\,G_2(b)-G_4(b) ,
\end{align}
see (\ref{e49}) and (\ref{e73}). Since $G_2(b)$ is expressed in terms of $\zeta$-functions in (\ref{e55}), it is sufficient to consider $G_4(b)$.

As to $G_4(b)$,
\beq \label{e75}
G_4(b)=G_5(b)-G_6(b) ,
\eq
where
\begin{align} \label{e76}
G_5(b)&=\il_0^{\infty}\,\frac{\exp({-}b^2-t^2)}{(1-\exp({-}b^2-t^2))^2}\,dt ,\\
 \label{e77}
G_6(b)&=\il_0^{\infty}\,\frac{b^2}{b^2+t^2}~\frac{\exp({-}b^2-t^2)}{(1-\exp({-}b^2-t^2))^2}\,dt .
\end{align}
We have, compare (\ref{e51}),
\begin{align} \label{e78}
G_5(b) & =  \sum_{k=0}^{\infty}\,(k+1)\,\il_0^{\infty}\,e^{-(k+1)(b^2+t^2)}\,dt \nonumber \\[3.5mm]
& =  \frac{\sqrt{\pi}}{2}\,e^{-b^2}\,\Phi(e^{-b^2},{-}\tfrac12,1) =  \frac{\pi}{4b^3}+\frac{\sqrt{\pi}}{2}\,\sum_{r=0}^{\infty}\,\zeta({-}\tfrac12-r)\,\frac{({-}1)^r\,b^{2r}}{r!} ,
\end{align}
the last identity being valid when $0<b<\sqrt{2\pi}$. Next we have, compare (\ref{e53}),
\begin{align} \label{e79}
G_6(b) & =  \sum_{k=0}^{\infty}\,(k+1)\,b^2\,\il_0^{\infty}\,\frac{\exp({-}(k+1)(b^2+t^2))}{b^2+t^2}\,dt \nonumber \\
& =  \frac{\pi}{2}\,b\,\sum_{k=0}^{\infty}\,(k+1)\,{\rm erfc}(b\,\sqrt{k+1}) .
\end{align}
From \cite[(5.4) and (5.21)]{jllerch} we have
\begin{eqnarray} \label{e80}
\sum_{n=1}^{\infty}\frac{n}{\sqrt{2\pi}}\il_{\beta\sqrt{n}}^{\infty}e^{-x^2/2}\,dx = \frac{3}{4\beta^4}-\frac{1}{24}-\frac{1}{\sqrt{2\pi}}\sum_{r=0}^{\infty}\frac{\zeta({-}3/2-r)({-}1/2)^r}{r!\,(2r+1)}\,\beta^{2r+1}
\end{eqnarray}
when $0<\beta<2\sqrt{\pi}$. Taking $\beta=b\,\sqrt{2}$ in (\ref{e80}), we get
\beq \label{e81}
G_6(b)=\frac{3\pi}{16b^2}-\frac{\pi b}{24}-\sqrt{\pi}\,\sum_{r=0}^{\infty}\frac{\zeta({-}3/2-r)({-}1)^r}{r!\,(2r+1)}\,b^{2r+2}
\eq
when $0<b<\sqrt{2\pi}$. The two results (\ref{e78}) and (\ref{e81}) can be combined, as in \cite[end of \S 5]{jllerch} and this yields
\beq \label{e82}
G_4(b)=\frac{\pi}{16b^3}+\frac{\pi b}{24}+\tfrac12\,\zeta({-}1/2)\,\sqrt{\pi}+\sqrt{\pi}\,\sum_{r=0}^{\infty}\frac{\zeta({-}3/2-r)({-}1)^r\,b^{2r+2}}{r!\,(2r+1)(2r+2)}
\eq
when $0<b<\sqrt{2\pi}$.

Finally, we can use (\ref{e82}) in (\ref{e74}), and we obtain with (\ref{e55}), for $0<b<\sqrt{2\pi} $,
\begin{align} \label{e83}
G_3(b) = &  \frac{\pi}{16b^3}-\frac{\pi}{8b}-\frac{\pi b}{24}-\zeta({-}1/2)\,\sqrt{\pi}  -~2\sqrt{\pi}\,\sum_{r=0}^{\infty}\,\frac{\zeta({-}3/2-r)({-}1)^r\,b^{2r+2}}{r!\,(2r+1)(2r+2)(2r+3)}.
\end{align}
The right-hand side of (\ref{e83}) equals the right-hand side of \cite[(2.3)]{jllerch} multiplied by ${\pi}/{(2b)}$ with $\beta=b\,\sqrt{2}$. Again, \cite[\S 7]{jllerch} and \cite[Theorem~2]{cumulants} yield an alternative infinite series expression for the series in (\ref{e83}) that converges for all $b>0$.

\subsection{Heavy-traffic limits for the variance}\label{sec4}
We have from (\ref{e38}) in \S\ref{sec1}, using the same approach and notation as in \S\ref{subsec3.1} for $\mu_Q$, that $\sigma_Q^2$ is given with exponentially small error by
\beq \label{e84}
\frac{-s\,\eta}{2\pi i}\,\il_{-\frac12\delta}^{\frac12\delta}\,\frac{v\,z(v)}{(z(v)-1)^2}~\frac{B\,\exp({-}\frac12\,s\,\eta\,v^2)}{1-B\,\exp({-}\frac12\,s\,\eta\,v^2)}\,dv,
\eq
with $B$ and $\eta$ given in (\ref{e42}). From $z({-}v)=z^{\ast}(v)$ for real $v$ we now compute
\beq \label{e85}
\frac{z(v)}{(z(v)-1)^2}-\frac{z({-}v)}{(z({-}v)-1)^2}={-}2i\,\frac{|z(v)|^2-1}{|z(v)-1|^4}\,{\rm Im}(z(v)) ,
\eq
and so (\ref{e84}) becomes
\beq \label{e86}
\frac{s\eta}{\pi}\,\il_0^{\frac12\delta}\,\frac{|z(v)|^2-1}{|z(v)-1|^4}\,v\,{\rm Im}(z(v))\,\frac{B\,\exp({-}\frac12\,s\,\eta\,v^2)}{1-B\,\exp({-}\frac12\,s\,\eta\,v^2)}\,dv .
\eq
From
\beq \label{e87}
{\rm Im}(z(v))=v+O(v^3) ,~~~~~~|z(v)|^2-1=z_{\rm sp}^2-1+O(v^2) ,
\eq
%NIEUW
we get for the expression in \eqref{e86}
\beq \label{y70}
\frac{s\eta}{\pi}\,\il_0^{\frac{1}{2}\delta}\,\frac{v^2\,(z_{\rm sp}^2-1+O(v^2))(1+O(v^2))}{((z_{\rm sp}-1)^2+v^2 + O((z_{\rm sp}-1)\,v^2)+O(v^4))^2}
\frac{B\,\exp({-}\frac12\,s\,\eta\,v^2)}{1-B\,\exp({-}\frac12\,s\,\eta\,v^2)}\,dv.
\eq
 When $2\alpha-1<0$, we have as for the case of $\mu_Q$ in \S\ref{subsec3.1} that the whole expression in \eqref{y70} is $O(\exp({-}b^2\,s^{1-2\alpha}))$ for any $b\in(0,b_0)$, as $s\to\infty$. When $2\alpha-1\geq 0$, we get as in the case of $\mu_Q$ after substitution $v = t\sqrt{{2}/{(s\,\eta})}$ for the expression in \eqref{y70}
\beq  \label{y71}
\frac{2}{\pi}\,\Bigl(\frac{s\,\eta}{2}\Bigr)^{3/2}\,\il_0^\infty\frac{t^2\,(z_{\rm sp}^2-1+O(t^2/s))(1+O(t^2/s))}{(d^2(s)+t^2)^2\,(1+O(1/s^{\alpha})+O(t^2/s))}~\frac{B\,e^{{-}t^2}}{1-B\,e^{{-}t^2}}\,dt
\eq
When $2\alpha-1\geq 0$, the leading order behavior of $\sigma_Q^2$ depends crucially on the factor $z_{\rm sp}^2-1+O(t^2/s)$. Here
\beq \label{y72}
z_{\rm sp}^2-1 = \frac{2\,\gamma\,\mu_X}{\sigma_X^2\,s^\alpha}\,\left(1+O(s^{-\alpha})\right)
\eq
is dominant when $\alpha<1$, while the $O(t^2/s)$ is dominant when $\alpha>1$. In the case that $\alpha\in(1/2,1)$, we get for the leading order behavior of $\sigma_Q^2$
\begin{align}
\frac{2}{\pi}\,\Bigl(\frac{s\,\eta}{2}\Bigr)^{3/2}& \,\frac{2\,\gamma\,\mu_X}{\sigma_X^2\,s^\alpha}\,\il_0^\infty\frac{t^2}{(d^2(s)+t^2)^2}\cdot~\frac{e^{{-}d^2(s)-t^2}}{1-e^{{-}d^2(s)-t^2}}\,dt\,\left(1+O(s^{\alpha-1})\right)\nonumber\\
&= \frac{\gamma\,\sigma_X}{\pi}\,\Bigl(\frac{2}{\mu_X}\Bigr)^{1/2}\,s^{3/2-\alpha}\,G_3(d(s))\,\left(1+O(s^{\alpha-1})\right), \label{y73}
\end{align}
where \eqref{e26}, \eqref{e27} and \eqref{e42} have been used for $\eta = g''(z_{\rm sp})$ and where $G_3$ is given in \eqref{e72}.

When we insert the expansion \eqref{e83} for $G_3(b)$, with the whole series on the second line being $O(b^2)$, we get the leading order behavior of $\sigma_Q^2$ as
\begin{align}
s^{2\alpha}\,\Bigl( \frac{\sigma_X^4}{4\,\gamma^2\mu_X^2}- \frac{\sigma_X^2}{4\,\mu_X}&\,\frac{1}{s^{2\alpha-1}} - \Bigl(\frac{2\,\sigma_X^2}{\pi\,\mu_X}\Bigr)^{1/2}\,\frac{\gamma\,\zeta(-1/2)}{s^{3\alpha-3/2}}\nonumber\\
& - \frac{\gamma^2}{24\,s^{5\alpha-5/2}}+O(s^{1-4\alpha})\Bigr)\,\left(1+O(s^{\alpha-1})\right)\nonumber \\
&\ = s^{2\alpha}\,\frac{\sigma_X^4}{4\,\gamma^2\,\mu_X^2}\,\Bigl(1+O(s^{\max(1-2\alpha,\alpha-1)})\Bigr)\label{y74}
\end{align}
when $\alpha\in(1/2,1)$. For the case $\alpha=1/2$, we get the leading order behavior, assuming $0<b_0<\sqrt{2\pi}$,
\begin{align}
\frac{\sigma_X^2 s}{\mu_X}\left[ \frac{1}{8\,b_0^2} - \frac{1}{4}-\frac{1}{12}\,b_0^2 - \frac{2\,\zeta(-1/2)}{\sqrt{\pi}}\,b_0- \frac{4}{\sqrt{\pi}}\,\sum_{r=0}^\infty \frac{\zeta(-3/2-r)\,(-1)^r\,b_0^{2r+3}}{r!\,(2r+1)\,(2r+2)\,(2r+3)} \right]\label{y75}
\end{align}
with relative error $O(s^{-1/2})$. The expression between brackets in \eqref{y75} coincides with the right-hand side of \cite{jllerch}, (2.3) with $\beta = b_0\,\sqrt{2}$.

This leads to the following two results.
\begin{theorem} \label{varthm}
For $\alpha\in[1/2,1)$,
\beq \label{y76}
\sigma_Q^2 = \frac{\gamma\,\sigma_X}{\pi}\,\sqrt{\frac{2}{\mu_X}}\,s^{3/2-\alpha}\,G_3(d(s))\,\left(1+O(s^{\alpha-1})\right)
\eq
with $G_3$ given in \eqref{e72}.
\end{theorem}

\begin{proposition}\label{varprop}
For $\alpha\in(0,1/2)$, and for all $b<b_0$,
\begin{equation}
\sigma_Q^2 = O(\exp({-}b^2\,s^{1-2\alpha})).
 \end{equation}
 For $\alpha = 1/2$, $\sigma_Q^2$ equals expression \eqref{y75} with relative error $O(s^{-1/2})$. For $\alpha\in(1/2,1)$ and $b_0\in(0,\sqrt{2\pi})$, $\sigma_Q^2$ has the form in \eqref{y74}.
\end{proposition}

As in \S\ref{subsec3.3} for the mean congestion level
with $\alpha=1/2$, it is possible to give a correction term which involves now integrals and series with $\zeta$-functions as considered in \cite[Secs.~4-5]{cumulants}.

\subsection{Heavy-traffic limits for the empty-system probability} \label{sec5}

We have from (\ref{e6}) by proceeding as in (\ref{e13})--(\ref{e17}) that
\begin{align} \label{e100}
{\rm ln}\,[Q(0)] & =  \frac{s}{2\pi i}\,\il_{|z|=1+\eps}\,{\rm ln}\Bigl(\frac{z}{z-1}\Bigr)\,\frac{g'(z)\,\exp(s\,g(z))}{1-\exp(s\,g(z))}\,dz \nonumber \\[3.5mm]
& =  \frac{1}{2\pi i}\,\il_{|z|=1+\eps}\,\frac{1}{z(z-1)}\,{\rm ln}\left(1-\exp(s\,g(z))\right)\,dz ,
\end{align}
where in the last step we used partial integration (noting that ${\rm Re}\,[g(z)]<0$ on $|z|=1+\eps$). Then, as in \S\ref{sec1} for $\mu_Q$, the last integral in (\ref{e100}) is, with exponentially small error, given by
\beq \label{e101}
\frac{1}{2\pi i}\,\il_{-\frac12\delta}^{\frac12\delta}\,\frac{z'(v)}{z(v)(z(v)-1)}\,{\rm ln}\left(1-B\,e^{-\frac12 s\eta v^2}\right)\,dv .
\eq
Now for $v\geq0$ from $z({-}v)=z^{\ast}(v)$, $z'({-}v)={-}(z'(v))^{\ast}$
\begin{eqnarray} \label{e102}
& \mbox{} & \hspace*{-6mm}\frac{z'(v)}{z(v)(z(v)-1)}+\frac{z'({-}v)}{z({-}v)(z({-}v)-1)}=2i\,{\rm Im}\,\Bigl[\frac{z'(v)}{z(v)(z(v)-1)}\Bigr] \nonumber \\[3.5mm]
& & \hspace*{-6mm}=~2i\,{\rm Im}\,\Bigl[\frac{z'(v)\,z^{\ast}(v)(z^{\ast}(v)-1)}{|z(v)|^2\,|z(v)-1|^2}\Bigr] \nonumber \\[3.5mm]
& & \hspace*{-6mm}=~2i\,\frac{z_{\rm sp}-1+O(v^2)}{(z_{\rm sp}+O(v^2))((z_{\rm sp}-1)^2+v^2-2c_2(z_{\rm sp}-1)\,v^2+O(v^4))}\,,
\end{eqnarray}
where we used \eqref{e32} and the fact that $z_{\rm sp}$ and $c_k$ are real with $z_{\rm sp}>1$. Therefore, we get for the expression in \eqref{e101}
\beq \label{y77}
\frac{1}{\pi}\il_0^{\frac{1}{2}\delta}\frac{1}{z_{\rm sp}{\rm +}O(v^2)}\frac{z_{\rm sp}-1+O(v^2)}{(z_{\rm sp}-1)^2+v^2+O((z_{\rm sp}-1)v^2)+O(v^4)}{\rm ln}\left(1-B\exp(-\tfrac12 s\eta v^2)\right)dv.
\eq
In the case that $2\alpha-1<0$, we have as earlier that the whole expression in \eqref{y77} is $O(\exp({-}b^2\,s^{1-2\alpha}))$ for any $b\in(0,b_0)$, as $s\to\infty$. In the case that $2\alpha-1\geq 0$, we substitute $v=t\sqrt{{s}/{(2\,\eta)}}$, and we get as earlier for the expression \eqref{y77}, assuming also that $\alpha<1$,
\begin{align}
\frac{1}{\pi}&\,\sqrt{s\,\eta/2}\,\il_0^{\infty}\frac{z_{\rm sp}-1+O(t^2/s)}{(d^2(s)+t^2)\,(1+O(s^{-\alpha})+O(t^2/s))}\,{\rm ln}(1-B\,e^{-t^2})dt\nonumber\\
&= \frac{1}{\pi}\,\il_0^{\infty}\frac{\sqrt{s\,\eta/2} \ (z_{\rm sp}-1)}{d^2(s)+t^2}{\rm ln}(1-B\,e^{-t^2})dt\,\left(1+O(s^{\alpha-1})\right)\nonumber\\
&= \frac{1}{\pi}\,\il_0^{\infty}\frac{d(s)}{d^2(s)+t^2}{\rm ln}(1-e^{{-}d^2(s)-t^2})dt\,\left(1+O(s^{\alpha-1})\right).
\label{y78}
\end{align}
Here we also used \eqref{y46} and that $1/s^{3\alpha-1} = O(d^2(s)/s^\alpha)$, so that
\beq \label{y79}
(\tfrac12\,s\,\eta)^{1/2}\,(z_{\rm sp}-1) = d(s)\,\left(1+O(s^{-\alpha})\right) = d(s)\left(1+O(s^{\alpha-1})\right),
\eq
since $\alpha\geq 1/2$.

We have for $b>0$
\begin{align}
\frac{1}{\pi}\,&\il_0^\infty \frac{b}{b^2+t^2}\,{\rm ln}(1-\exp({-}b^2-t^2))\,dt =-\frac12\,\sum_{k=0}^{\infty}\,\frac{1}{k+1}\,{\rm erfc}(b\,\sqrt{k+1}) = -F(b\,\sqrt{2}),\label{y80}
\end{align}
where according to \cite[(3.3) and (3.12)]{jllerch} for $\beta>0$
\begin{align}
F(\beta) &= \sum_{n=1}^\infty\,\frac{1}{n}\,\frac{1}{\sqrt{2\pi}}\,\il_{\beta\sqrt{n}}^\infty e^{-x^2/2}dx\nonumber\\
&= -{\rm ln}\,\beta - \frac12\,{\rm ln}2 - \frac{1}{\sqrt{2\pi}}\,\sum_{r=0}^\infty \frac{\zeta(1/2-r)\,(-1/2)^r\,\beta^{2r+1}}{r!\,(2r+1)},\label{y81}
\end{align}
the last identity being valid for $0<\beta<2\sqrt{\pi}$.

Using \eqref{y81} with $\beta^2 = d^2(s)= b_0^2/s^{2\alpha-1}$, with the entire series on the second line being $O(\beta)$, we get the leading order behavior of ${\rm ln}[Q(0)]$ as
\beq \label{y82}
\Bigl({-}(\alpha-1/2)\,{\rm ln}\,s+{\rm ln}(2\,b_0)+O(s^{1/2-\alpha})\Bigr)\left(1+O(s^{\alpha-1})\right)
\eq
when $\alpha\in(1/2,1)$. For $\alpha = 1/2$, we get the leading order behavior, assuming $0<b_0<\sqrt{2\pi}$,
\beq \label{y83}
{\rm ln}(2\,b_0) + \frac{1}{\sqrt{\pi}}\,\sum_{r=0}^\infty \,\frac{\zeta(1/2-r)\,(-1)^r}{r!\,(2r+1)}\,b_0^{2r+1}
\eq
with relative error $O(s^{-1/2})$. The expression \eqref{y83} coincides with ${\rm ln}[\mathbb{P}(M=0)]$ as given by \cite[(2.1)]{jllerch} with $\beta = b_0\,\sqrt{2}$. The next two results summarize the above.

\begin{theorem} \label{emptythm}
For $\alpha\in(1/2,1)$,
\beq \label{y84}
{\rm ln}[\mathbb{P}(Q=0)] = - F\big(d(s)\,\sqrt{2}\big)\left(1+O(s^{\alpha-1})\right)
\eq
with $F$ given by \eqref{y81}.
\end{theorem}
\begin{proposition} \label{emptyprop}
For $\alpha\in (0,1/2)$, and for all $b<b_0$,
\begin{equation}
{\rm ln}[\mathbb{P}(Q=0)] = O(\exp({-}b^2\,s^{1-2\alpha})).
 \end{equation}
 For $\alpha=1/2$, ${\rm ln}[\mathbb{P}(Q=0)]$ equals $-F(b_0\,\sqrt{2})$ with a relative error $O(1/\sqrt{s})$. For $\alpha\in (1/2,1)$ and $0<b_0<\sqrt{2\pi}$, ${\rm ln}[\mathbb{P}(Q=0)]$ has leading order behavior as in \eqref{y82}.
\end{proposition}

As in \S\ref{subsec3.3} for the mean congestion level case
with $\alpha=1/2$, it is possible to give a correction term which involves now the integrals in \eqref{y80} and \eqref{e51}.
%
%We briefly indicate how to obtain for $\alpha = 1/2$ a first-order correction term. It appears that including more precise approximations in the integral in (\ref{e101}) leads to the integrals
%\beq \label{e108}
%\il_0^{\infty}\,\frac{1}{(b_0^2+t^2)^2}\,{\rm ln}(1-\exp({-}b_0^2-t^2))\,dt ,
%\eq
%\beq \label{e109}
%\il_0^{\infty}\,\frac{1}{b_0^2+t^2}~\frac{\exp({-}b_0^2-t^2)}{1-\exp({-}b_0^2-t^2)}\,dt .
%\eq
%The integral in (\ref{e109}) equals $b_0^{-2}\,G_2(b_0)$, see (\ref{e50}) and (\ref{e55}). For the integral in (\ref{e108}) we proceed as follows. We have
%\begin{eqnarray} \label{e110}
%& \mbox{} & \il_0^{\infty}\,\frac{{\rm ln}(1-\exp({-}b_0^2-t^2))}{(b_0^2+t^2)^2}\,dt \nonumber \\[3.5mm]
%& & =~\frac{1}{b_0^2}\,\Bigl[\il_0^{\infty}\,\frac{{\rm ln}(1-\exp({-}b_0^2-t^2))}{b_0^2+t^2}\,dt-\il_0^{\infty}\,\frac{t^2\,{\rm ln}(1-\exp({-}b_0^2-t^2))}{(b_0^2+t^2)^2}\,dt\Bigr] \nonumber \\[3.5mm]
%& & =~\frac{1}{b_0^2}\,\Bigl[\tfrac12\,\il_0^{\infty}\,\frac{{\rm ln}(1-\exp({-}b_0^2-t^2))}{b_0^2+t^2}\,dt-\il_0^{\infty}\,\frac{t^2}{b_0^2+t^2}~\frac{\exp({-}b_0^2-t^2)}{1-\exp({-}b_0^2-t^2)}\,dt\Bigr], \nonumber \\[2mm]
%\end{eqnarray}
%where in the last step partial integration has been used. The remaining two integrals have already been evaluated, see (\ref{y80}) and (\ref{e49}).

\section{Numerical examples}\label{numm}
\subsection{Accuracy of the approximations}
In this subsection we present a numerical example that serves to illustrate the accuracy of the derived heavy-traffic approximations. Consider the Poisson case
\beq
X(z)=e^{z-1},\quad \mu_X = \sigma_X^2 = 1.
\eq
We fix $\mu_X$ and vary $n$ with the value of $s$, according to
\beq
\vart = \frac{n}{s} = 1-\frac{\gamma}{s^\alpha}
\eq
for some $\gamma>0$ and $\alpha\geq 1/2$. To calculate the exact value of the mean congestion level we use the expression, see \cite{boudreau},
\eqan{\label{x73}
\mu_Q=\frac{\sigma_A^2}{2(s-\mu_A)}-\frac{s-1+\mu_A}{2}+\sum_{k=1}^{s-1}\frac{1}{1-z_k}.
}
Here $z_1,\ldots,z_{s-1}$ are the zeros of $z^s-A(z)$ in $|z|<1$. We apply the method of successive substitution described in \cite{ref11} to obtain accurate numerical approximations for $z_1,...,z_{s-1}$ and consequently $\mu_Q$. 

From Theorem \ref{mainthm}, we find that the leading order behavior of $\mu_Q$ is given by
\beq \label{x18}
\frac{\sqrt{2s}}{\pi}\,G_0\Bigl(\frac{\gamma}{\sqrt{2}\,s^{\alpha-\frac{1}{2}}}\Bigr).
\eq
In order to find the correction terms, we proceed by setting $\alpha = 1/2$. Deriving constants $C_1,C_2,C_3,$ and $C_4$ for our setting and substituting these into \eqref{e71},  we get for $\mu_Q$, with an absolute error of $O(s^{-1/2})$, the approximation
\beq\label{x19}
\frac{\sqrt{2\,s}}{\pi}\Bigl(\Bigl(1-\frac{\gamma}{3\,\sqrt{s}}\Bigr)\,G_0(b_0)-\frac{\gamma^3}{3\,\sqrt{s}}\,(\,G_3(b_0)+G_4(b_0))\Bigr),
\eq
which by \eqref{e49} and \eqref{e74} reduces to
\beq\label{x20}
\frac{\sqrt{2\,s}}{\pi}\,G_0(b_0)-\frac{\sqrt{2}\,\gamma}{3\,\pi}\,G_1(b_0).
\eq
\begin{table}
\parbox{.47\linewidth}{
\centering
\footnotesize
\begin{tabular}{rrrrr}
$s$ & $\rho$ & $\mu_Q$ & \eqref{x18} & \eqref{x20}\\
\hline
 10 & 0.683 & 0.244 & 0.399 & 0.247 \\
 20 & 0.776 & 0.410 & 0.565 & 0.412 \\
 50 & 0.858 & 0.739 & 0.893 & 0.741 \\
 100 & 0.900 & 1.110 & 1.263 & 1.111 \\
 200 & 0.929 & 1.633 & 1.787 & 1.634 \\
 500 & 0.955 & 2.672 & 2.825 & 2.673 \\
 1000 & 0.968 & 3.843 & 3.996 & 3.843\label{tab:poisson1}
\end{tabular}
\caption{Numerical results for $\gamma = 1$.}}
\hfill
\parbox{.48\linewidth}{
\centering
\footnotesize
\begin{tabular}{rrrrr}
$s$ & $\rho$ & $\mu_Q$ & \eqref{x18} & \eqref{x20}\\
\hline
 10 & 0.968 & 13.707 & 14.046 &13.732\\
 20 & 0.977 & 19.533 & 19.865 &19.551\\
 50 & 0.985 & 31.084 & 31.409 &31.095\\
 100 & 0.990 & 44.097 & 44.419 &44.106\\
 200 & 0.992 & 62.499 & 62.819 &62.505\\
 500 & 0.995 & 99.008 & 99.325 &99.011\\
 1000 & 0.996 & 140.152 & 140.468 &140.154\label{tab:poisson2}
\end{tabular}
\caption{Numerical results for $\gamma = 0.1$.}}
\end{table}
\begin{table}
\centering
\footnotesize
\begin{tabular}{rrr|rr|rr}
 & \multicolumn{2}{c}{$\alpha=0.6$} & \multicolumn{2}{|c|}{$\alpha=0.75$} & \multicolumn{2}{c}{$\alpha=0.9$}\\
$s$ &  $\mu_Q$ & \eqref{x18} &  $\mu_Q$ & \eqref{x18} &  $\mu_Q$ & \eqref{x18}\\
\hline
 10 & 17.781 & 18.125 & 25.970 & 26.318 & 37.553 & 37.905 \\
 20 & 27.309 & 27.647 & 44.391 & 44.734 & 71.195 & 71.541 \\
 50 & 47.948 & 48.281 & 89.623 & 89.961 & 164.637 & 164.978 \\
 100 & 73.245 & 73.574 & 152.031 & 152.367 & 309.353 & 309.692 \\
 200 & 111.752 & 112.079 & 257.435 & 257.769 & 580.170 & 580.507 \\
 500 & 195.082 & 195.409 & 515.443 & 515.776 & 1329.581 & 1329.917 \\
 1000 & 297.122 & 297.448 & 870.524 & 870.857 & 2487.227 & 2487.562\label{tab:poisson3}
\end{tabular}
\caption{Numerical results for $\gamma=0.1$ and several values of $\alpha$.}
\end{table}
\noindent Numerical results for $\alpha=1/2$ and various values of $s$ are given in Table 1 and 2, for $ \gamma = 1$ and  $\gamma = 0.1$, respectively.
We note that for small $s$ the leading order approximation is still off by a significant amount, while the refinement only shows an error in the second decimal for $\gamma = 0.1$. This seems to justify the use of the correction term. 
In Table 3 we compare the approximation \eqref{x18} against the exact value of $\mu_Q$ for three values of $\alpha\geq 1/2$ to assess the influence of $\alpha$. Clearly, the leading order approximation is relatively accurate for all three scenarios. As expected, the mean congestion increases along with $\alpha$, since utilization approaches 1 more rapidly in this case. 
\subsection{Connection to other queueing models}\label{subsec62}
As argued in the introduction, we believe that the heavy-traffic behavior for our discrete model will up to leading order be universal for a wide range of other models (when subjected to the same heavy traffic regime \eqref{bb}). We shall now substantiate this for many-server systems, for which under \eqref{bb}, it turns out that the mean congestion is $O(s^\alpha)$.  We compare the mean congestion level in our discrete queue with that in the multi-server systems $M/M/s$, $M/D/s$ and Gamma/Gamma/$s$, all with unit mean service time and occupation rate $1-\gamma/s^\alpha$. 

\begin{figure}
\centering
\begin{minipage}{.49\textwidth}
  \centering
  \includegraphics[width=.9\linewidth]{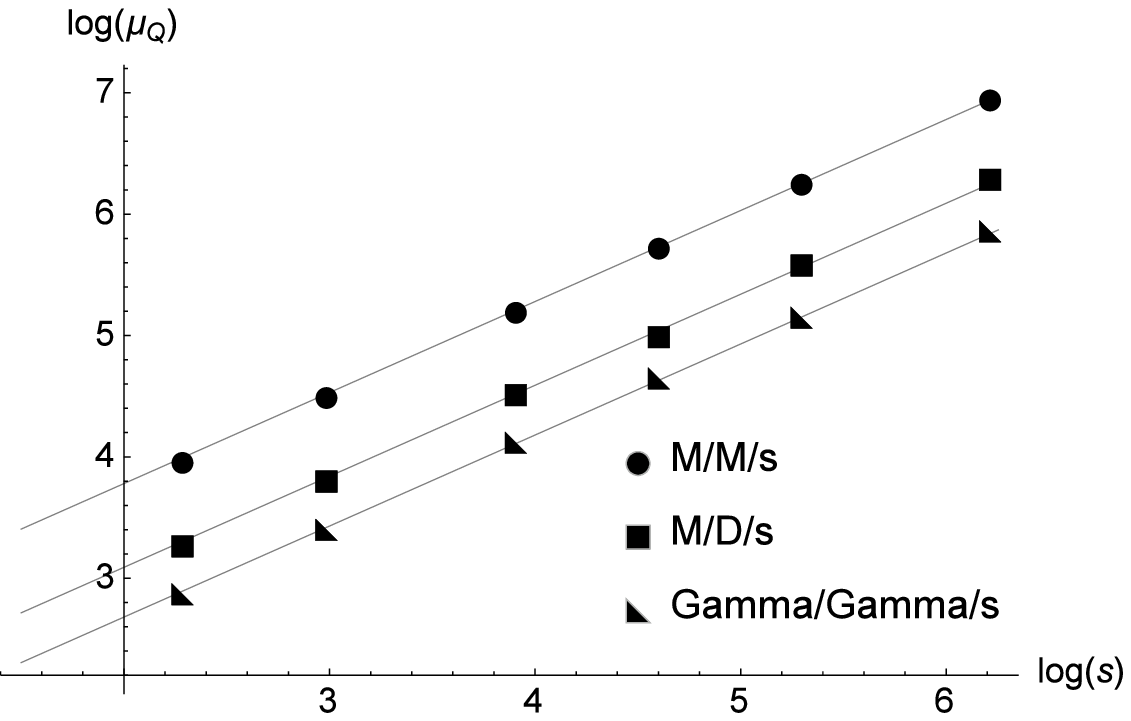}
  \captionof{figure}{$\mu_Q$ plotted against $s$ on log scale for 3 queues for $\alpha=0.75$.}
  \label{fig1}
\end{minipage}\hfill
\begin{minipage}{.49
\textwidth}
  \centering
  \includegraphics[width=.9\linewidth]{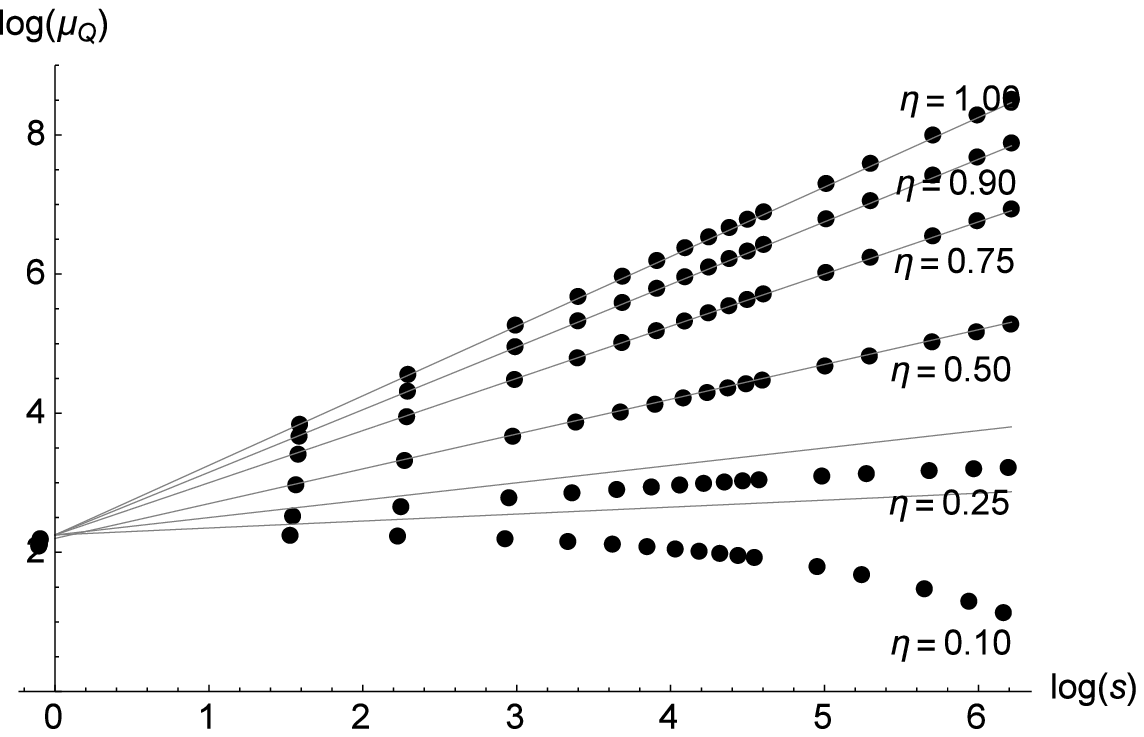}
  \captionof{figure}{$\mu_Q$ of $M/M/s$ plotted against $s$ on log scale for different values of $\alpha$.}
  \label{fig2}
\end{minipage}
\end{figure}
Figure \ref{fig1} shows on logarithmic scale the mean congestion levels for $\gamma=0.1$ and  $\alpha=0.75$  under the specified scaling for three systems. We also display three lines with slope 0.75 for comparison, which confirms that mean congestion levels are of the order $s^\alpha$, also in these multi-server systems. Formally establishing this heavy-traffic behavior for these multi-server system is an important open problem and requires other mathematical approaches than the ones taken in this paper (see the introduction for more details). 

Figure \ref{fig2} shows the mean queue length in the $M/M/s$ system for several values of $\alpha$, again on logarithmic scale, together with lines with slope $\alpha$. For $\alpha\geq 1/2$, we see the same $O(s^\alpha)$ behavior, similar as for $\mu_Q$ in our discrete model. For $\alpha<1/2$ the mean queue length decreases, again in agreement with our results for $\mu_Q$. We note that this qualitative behavior of the $M/M/s$ system was also observed by \cite[Thm 4.1]{maman}, by proving that the mean waiting time in the $M/M/s$ queue under \eqref{bb} is of the order $1/s^{1-\alpha}$, which by Little's law implies that the mean queue length is of the order $s^\alpha$.

\appendix

\section{Proof of Pollaczek's formula in the discrete setting} \label{app}

In the setting of \S\ref{sec1}, we shall show that for any $\eps>0$ with $1+\eps<r_0$,
\beq \label{e111}
Q(w)=\exp\Bigl(\frac{1}{2\pi i}\,\il_{|z|=1+\eps}\,{\rm ln}\Bigl(\frac{w-z}{1-z}\Bigr)\,\frac{(z^s-A(z))'}{z^s-A(z)}\,dz\Bigr)
\eq
holds when $|w|<1+\eps$. We shall establish (\ref{e111}) for any $w\in(1,1+\eps)$, and then the full result follows from analyticity of $Q(w)$ and of
\beq \label{e112}
{\rm ln}\Bigl(\frac{w-z}{1-z}\Bigr)={\rm ln}\Bigl(\frac{1-w/z}{1-1/z}\Bigr)={-}\,\sum_{k=1}^{\infty}\,\frac1k\,\Bigl(\Bigl(\frac{w}{z}\Bigr)^k-\Bigl(\frac1z\Bigr)^k\Bigr)
\eq
in $w$, $|w|<1+\eps$ for any $z$ with $|z|=1+\eps$.

Our starting point is the formula, see \cite{boudreau},
\beq \label{e113}
Q(w)=\frac{(s-\mu_A)(w-1)}{w^s-A(w)}\,\prod_{k=1}^{s-1}\,\frac{w-z_k}{1-z_k}
\eq
that holds for all $w$, $|w|<r_0$, in which $z_1,\ldots,z_{s-1}$ are the $s-1$ zeros of $z^s-A(z)$ in $|z|<1$. Fix $w\in(1,1+\eps)$. Then ${\rm ln}\,[(w-z)/(1-z)]$ is analytic in $z\in\dC\backslash [1,w]$. It follows that
\begin{align}
 I_C &= \frac{1}{2\pi i}\,\il_{|z|=1+\eps}\,{\rm ln}\Bigl(\frac{w-z}{1-z}\Bigr)\,\frac{(z^s-A(z))'}{z^s-A(z)}\,dz \nonumber \\
&=~\sum_{k=1}^{s-1}\,{\rm ln}\Bigl(\frac{w-z_k}{1-z_k}\Bigr)+\frac{1}{2\pi i}\,\il_C\,{\rm ln}\Bigl(\frac{w-z}{1-z}\Bigr)\,\frac{(z^s-A(z))'}{z^s-A(z)}\,dz ,
\label{e114}
\end{align}
where $C$ is a contour encircling $[1,w]$ in the positive sense with none of the $z_k$'s in its interior. We let $\delta\in(0,\frac{w-1}{2})$ and we take $C$ the union of two line segments, from $1+\delta-i0$ to $w-\delta-i0$ and from $w-\delta+i0$ to $1+\delta-i0$, and two circles, of radius $\delta$ and encircling 1 and $w$ in positive sense. 
A careful administration of the various contributions to the integral $I_C$ in \eqref{e114}, taking account of the branch cut $[1,w]$, yields
\begin{equation}\label{e115}
I_C = {\rm ln }\left(\frac{(s-\mu_A)(w-1)}{w^s-A(w)}\right) + O(\delta\,\rm{ln}\,\delta). 
\end{equation}
Using this in \eqref{e113} and letting $\delta \downarrow 0$, we get \eqref{e111} for $w\in(1,1+\varepsilon)$ and the proof is complete.\newline
\newline

\bibliography{Bibliography}

\end{document}